\documentclass{amsart}
\usepackage{amsmath}
\usepackage{amsfonts}
\usepackage{amsthm}
\usepackage{amssymb}

\newtheorem{theorem}{Theorem}[section]
\newtheorem*{theorem*}{Theorem}
\newtheorem*{acknowledgement*}{Acknowledgements}
\newtheorem{conjecture}[theorem]{Conjecture}
\newtheorem{corollary}[theorem]{Corollary}
\newtheorem{lemma}[theorem]{Lemma}

\newtheorem{remark}[theorem]{Remark}
\numberwithin{equation}{section}

\newcommand{\abs}[1]{\left\vert #1\right\vert}
\newcommand{\dist}{\operatorname{dist}}
\newcommand{\dive}{\operatorname{div}}
\newcommand{\e}{\operatorname{e}}
\newcommand{\ee}{\mathbf{e}}
\newcommand{\grad}{\nabla}
\newcommand{\HH}{\mathbb{H}}
\newcommand{\la}{\langle}
\newcommand{\lap}{\operatorname{\Delta}}

\newcommand{\nn}{\mathbf{n}}
\newcommand{\NN}{\mathbf{N}}

\newcommand{\pp}{\mathbf{p}}
\newcommand{\PP}{\mathbb{P}}
\newcommand{\ra}{\rangle}
\newcommand{\RR}{\mathbb{R}}
\newcommand{\spa}{\operatorname{span}}
\newcommand{\vv}{\mathbf{v}}
\newcommand{\xx}{\mathbf{x}}
\newcommand{\yy}{\mathbf{y}}
\newcommand{\oo}{\mathbf{0}}

\title{Geometry of Two-dimensional Self-shrinkers}
\author{Lu Wang}
\address{Department of Mathematics, Imperial College London, South Kensington Campus, London, SW7 2AZ, UK}
\email{lwang@math.jhu.edu}
\thanks{The author was partly supported by the Chapman Fellowship of the Imperial College London and by the AMS-Simons Travel Grant 2012--2014 and the NSF Grant DMS-1406240}

\begin{document}
\begin{abstract}
We prove a local graphical theorem for $2$-dimensional self-shrinkers away from the origin. As applications, we study the asymptotic behavior of noncompact self-shrinkers with finite genus. Also, we show uniform boundedness on the second fundamental form of $2$-dimensional noncompact self-shrinkers with bounded mean curvature and uniform locally finite genus. 
\end{abstract}
\maketitle

\section{Introduction}
A surface $\Sigma\subset\RR^{3}$ is called a \textit{self-shrinker} if
\begin{equation}\label{eqn:shrinker}
H=\frac{1}{2}\left\la\xx,\nn\right\ra.
\end{equation}
Here $\nn$ is the unit normal to $\Sigma$, $H=\dive\nn$ is the mean curvature and $\xx$ is the position vector in $\RR^{3}$.

On one hand, self-shrinkers are critical points of the Gaussian surface area 
\begin{equation}\label{eqn:gaussarea}
\mathbf{F}[\Sigma]=\left(4\pi\right)^{-1}\int_\Sigma\e^{-\frac{\abs{\xx}^2}{4}}d\mathcal{H}^2,
\end{equation}
where $\mathcal{H}^2$ is the $2$-dimensional Hausdorff measure on $\RR^{3}$. Note that $\mathbf{F}[\Sigma]$ is finite if and only if $\Sigma$ is properly immersed, cf. \cite[Theorem 1.3]{CZ} and \cite[Theorem 1.1]{DX}. On the other, self-shrinkers are a special class of solutions of mean curvature flow in which a later time slice is a scale-down copy of an earlier one. Combining the monotonicity formula of Huisken \cite{H} with the compactness theorem of Brakke \cite{Ba}, Ilmanen in \cite{I1} showed that the singularities of mean curvature flow of surfaces with bounded area ratio and genus are modeled by properly embedded self-shrinkers. 

One of the most important questions in the study of mean curvature flow is on the classification of self-shrinkers. Various examples of self-shrinkers (cf. \cite{Ch}, \cite{KKM} and \cite{Ng}) indicate that the classification problem in general could be very difficult. However, there has been significant progress in this direction since the 1990s; see \cite{Be}, \cite{CM2}, \cite{CIM}, \cite{EH1}, \cite{H}, \cite{KM} and \cite{Wa1}. Recently it has been proven in \cite{Wa2,Wa3} that two self-shrinkers asymptotic along some end of each to the same cone or cylinder in certain sense must identically coincide with each other. It is such a rigidity theorem that motivates us to investigate the geometry of self-shrinkers at infinity.

Indeed it is a conjecture (see page 39 of \cite{I2}) that

\begin{conjecture}\label{con:asymptotic}
Suppose that $\Sigma\subset\RR^3$ is a complete noncompact properly embedded self-shrinker with finite genus\footnote{The conditions on $\Sigma$ are implied from the context of \cite{I2} and \cite[Theorem 1.3]{CZ}.}. Then $\Sigma$ outside some closed ball decomposes into a finite number of connected components $U_i$ such that for each $i$, either of the following statements holds:
\begin{enumerate}
\item as $\lambda\to\infty$, $\lambda^{-1}U_i$ converges locally smoothly to a cone which is smooth except at the origin $\oo$.
\item there exists a unit vector $\vv_i$ such that, as $\lambda\to\infty$, $U_i-\lambda\vv_i$ converges locally smoothly to the self-shrinking cylinder with axis parallel to $\vv_i$.
\end{enumerate}
\end{conjecture}

For self-shrinkers of finite topology, Conjecture \ref{con:asymptotic} can be reformulated as a question on the singularities formation of self-shrinking solutions of mean curvature flow at time $0$. Namely, for any self-shrinker $\Sigma\subset\RR^3$, $\Sigma_t=\sqrt{-t}\, \Sigma$ for $t<0$ moves by mean curvature. It is known by \cite{I2} that, as $t\to 0$, $\Sigma_t$ converges in the Hausdorff metric to a cone consisting of points at which the Gaussian densities of $\{\Sigma_t\}_{t<0}$ are nonzero. Given $\vv_0\neq\oo$ on the cone, the one-parameter family of surfaces, $\Sigma_s=\Sigma-\e^{s/2}\vv_0$, is the corresponding normalized mean curvature flow centered at $(\vv_0,0)$. In particular the Gaussian surface area is non-increasing along the flow $\{\Sigma_s\}_{s\in\RR}$. Thus, to address Conjecture \ref{con:asymptotic}, it suffices to prove:

\begin{conjecture}\label{con:converge}
Suppose that $\Sigma\subset\RR^3$ is a complete noncompact properly embedded self-shrinker with finite genus. Then given $\vv_0\in\RR^3\setminus\{\oo\}$ such that $\mathbf{F}[\Sigma-\lambda\vv_0]\geq 1$ for all $\lambda>0$, as $\lambda\to\infty$, $\Sigma-\lambda\vv_0$ converges locally smoothly to a plane or a self-shrinking cylinder of multiplicity one.
\end{conjecture}

By the stratification theorem of White \cite{Wh}, given a sequence $\lambda_i\to\infty$, there is a subsequence $\lambda_{i_j}\to\infty$ such that $\Sigma-\lambda_{i_j}\vv_0$ in Conjecture \ref{con:converge} converges in the sense of measures to some multiple of a complete properly embedded self-shrinker in $\RR^3$ which splits off a line along the direction of $\vv_0$. Furthermore, by the classification \cite{AL} for $1$-dimensional self-shrinkers (see also \cite[Corollary 10.45]{CM2}), the limiting self-shrinker must be a plane or a self-shrinking cylinder. Thus one main difficulty to address Conjecture \ref{con:converge} is on the smooth convergence which is the goal of this paper. To achieve this, we need to establish a local graphical decomposition theorem for $2$-dimensional self-shrinkers. 

Let $B_r(\xx)\subset\RR^3$ be the open ball centered at $\xx$ with radius $r$ and if the center is the origin, we will omit it. Also, let $A$ be the second fundamental form.

\begin{theorem}\label{thm:sheeting}
Given $\kappa_0>0$ and $\delta_0\in (0,1)$, there exist constants $\epsilon_0,r_0\in (0,1)$ and $R_0>0$ such that for $\xx_0\in\RR^3\setminus B_{R_0}$ and a properly embedded self-shrinker $\Sigma$ in $B_1(\xx_0)$ with $\partial\Sigma\subset\partial B_1(\xx_0)$ and $\xx_0\in\Sigma$, if
\begin{equation}\label{eqn:slmc}
\abs{H(\xx)}<\epsilon_0\abs{\xx}\quad\mbox{for all $\xx\in\Sigma$},\quad\mbox{and}
\end{equation}
\begin{equation}\label{eqn:totalcurv}
\int_\Sigma\abs{A}^2d\mathcal{H}^2<\kappa_0,
\end{equation}
then the connected component of $\Sigma\cap B_{r_0}(\xx_0)$ containing $\xx_0$ is given by the graph of a function on a subset of the tangent plane $T_{\xx_0}\Sigma$ with gradient bounded by $\delta_0$.
\end{theorem}

The conditions \eqref{eqn:slmc} and \eqref{eqn:totalcurv} arise from the applications we will discuss in the next paragraph. The proof of Theorem \ref{thm:sheeting}, in part inspired by \cite{CIM}, relies on a bootstrap machinery and an observation that the radial direction on a self-shrinker can be also thought of as the time direction of the flow and so bears two distinct scaling properties.

We conclude the introduction by a couple of applications of Theorem \ref{thm:sheeting}. First we consider any noncompact properly embedded self-shrinker with finite genus. Given $\vv_0\in\RR^3\setminus\{\oo\}$, Huisken's monotonicity formula implies that the self-shrinker asymptotically splits off a line in the $L^2$ sense (see \eqref{eqn:intsplit}) in neighborhoods of $\lambda_i\vv_0$ for some sequence $\lambda_i\to\infty$. Then the conditions \eqref{eqn:slmc} and \eqref{eqn:totalcurv} of Theorem \ref{thm:sheeting} are  guaranteed by the linear growth of the second fundamental form \cite[Theorem 19]{So} (see also Appendix B) and Ilmanen's local Gauss-Bonnet estimate \cite[Theorem 3]{I1}. Thus we prove the following $C^1$ convergence theorem for the sequence of surfaces of translating the self-shrinker by $\lambda_i\vv_0$.

\begin{theorem}\label{thm:seqsheeting} 
Let $\Sigma\subset\RR^3$ be a complete noncompact properly embedded self-shrinker with finite genus. And let $\vv_0\in\RR^3\setminus\{\oo\}$ satisfy $\mathbf{F}[\Sigma-\lambda\vv_0]\geq 1$ for all $\lambda>0$. Then there exists an integer $N$ and two sequences $\lambda_i\to\infty$ with $\lambda_{i+1}/\lambda_i\to 1$ and $R_i\to\infty$ such that:
\begin{enumerate}
\item for each $i$,
\begin{equation}
(\Sigma-\lambda_i\vv_0)\cap B_{R_i}=\bigcup_{j=1}^N\Sigma_{i,j},
\end{equation}
where each $\Sigma_{i,j}$ is connected and $\Sigma_{i,j}\cap\Sigma_{i,k}=\emptyset$ if $j\neq k$; 
\item for each $i$ and $j$, there is a surface $\tilde{\Sigma}_{i,j}$, which is either a plane containing $\vv_0$ or the self-shrinking cylinder with axis parallel to $\vv_0$, and a connected domain $\Omega_{i,j}\subset\tilde{\Sigma}_{i,j}$ such that $\Sigma_{i,j}$ can be written as the normal graph of a function $u_{i,j}$ defined on $\Omega_{i,j}$;
\item for each $j$, as $i\to\infty$,
\begin{equation}
\sup_{\Omega_{i,j}}\abs{u_{i,j}}+\abs{\grad u_{i,j}}\to 0.
\end{equation}
\end{enumerate}
\end{theorem}

Moreover, in the asymptotically cylindrical case, the statement of Theorem \ref{thm:seqsheeting} can be strengthened as follows.

\begin{corollary}\label{cor:cylinder}
Let $\Sigma\subset\RR^3$ be a complete noncompact properly embedded self-shrinker with finite genus. And let $\vv_0\in\RR^3\setminus\{\oo\}$ satisfy that, as $\lambda\to\infty$, $\Sigma-\lambda\vv_0$ converges in the sense of measures to some multiple of the self-shrinking cylinder with axis parallel to $\vv_0$. Then the convergence is smooth and with multiplicity one.
\end{corollary}

Second, appealing to the local Gauss-Bonnet estimate, we apply Theorem \ref{thm:sheeting} to derive a uniform pointwise bound on the second fundamental form for an interesting class of self-shrinkers.

\begin{theorem}\label{thm:bmc}
Let $\Sigma\subset\RR^3$ be a complete noncompact properly embedded self-shrinker. Assume that there exist constants $\mathcal{H}_0, g_0>0$ such that 
\begin{equation}\label{eqn:bmc}
\sup_{\Sigma}\abs{H}<\mathcal{H}_0, \quad\mbox{and} 
\end{equation}
\begin{equation}\label{eqn:genus}
\sup_{\xx_0\in\RR^3} {\rm genus}(\Sigma\cap B_2(\xx_0))<g_0.
\end{equation}
Then $|A|\leq\mathcal{A}_0$ on $\Sigma$ for some constant $\mathcal{A}_0=\mathcal{A}_0(g_0,\mathcal{H}_0,\mathbf{F}[\Sigma])$. 
\end{theorem}

To some extent, Theorem \ref{thm:bmc} can be thought of as a mean curvature flow analog to \cite[Theorem 0.1]{MW} in which Munteanu-Wang proved that bounded scalar curvature for $4$-dimensional shrinking gradient Ricci solitons implies bounded curvature operator. However, the mean curvature here may change sign and the method employed in our proof is different from \cite{MW}. 

\begin{acknowledgement*}
The author thanks Neshan Wickramasekera for his hospitality and numerous helpful discussions during her visit of Cambridge University in May and June 2013 when part of this work was carried out. 
\end{acknowledgement*}

\section{Proof of Theorem \ref{thm:sheeting}}
The proof of Theorem \ref{thm:sheeting} is based on a bootstrap argument  to improve the oscillation of the unit normal $\nn$ to the given self-shrinker $\Sigma$. First the mean curvature bound \eqref{eqn:slmc} will be invoked in two ways. One is to show the same type of bound on the second fundamental form; see Lemma \ref{lem:slcurv}. The other is to yield directly from the self-shrinker equation \eqref{eqn:shrinker} that the position vector on $\Sigma$ is almost tangent to $\Sigma$. Next we give a brief explanation of the bootstrap machinery. The input is that for some $\epsilon\in (0,1)$, given $\xx\in\Sigma$, the oscillation of $\nn$ in the geodesic ball centered at $\xx$ with radius $1/(\epsilon |\xx|)$ is bounded by a fixed small constant. We may assume that $\epsilon^2|\xx|>1$ and let ${\sim}$ denote some multiple of the constant after it. Then we show in Lemma \ref{lem:radext} that for $\epsilon$ small and $|\xx|$ large, $\Sigma$ contains the graph of a $C^1$ function defined on the rectangle centered at $\xx$ with side lengths ${\sim}1/(\epsilon^2|\xx|)$ in the direction of $\xx$ and ${\sim}1/(\epsilon |\xx|)$ in the direction of $\nn(\xx)\times\xx$. Consequently the rate of the change of $\nn$ along the direction of $\xx$ is bounded by ${\sim}\epsilon^2|\xx|$; see Lemma \ref{lem:imposcnormal}. Thus the total curvature bound \eqref{eqn:totalcurv} gives the output that the oscillation of $\nn$ in the geodesic ball centered at $\xx$ with radius ${\sim}1/(\epsilon^2|\xx|)$ is bounded by the same constant as the input; see Lemma \ref{lem:iteration}. Hence this bootstrap machinery eventually leads to Theorem \ref{thm:sheeting}.

First we adopt a twisted argument of Choi-Schoen \cite{CS} to derive a pointwise bound on the second fundamental form of any properly embedded self-shrinker to which the unit normal is almost perpendicular to the position vector.

\begin{lemma}\label{lem:slcurv}
There exist constants $\epsilon_1\in (0,1)$ and $C_1>0$ such that for $0<\epsilon<\epsilon_1$ and $\xx_1\in\RR^3$ with $\epsilon|\xx_1|>1$, if $\Sigma$ is a properly embedded self-shrinker in $B_{r_1}(\xx_1)$ with $\partial\Sigma\subset\partial B_{r_1}(\xx_1)$ for $r_1=1/|\xx_1|$ which satisfies
\begin{equation}\label{eqn:smallmc}
\abs{H(\xx)}<\epsilon\abs{\xx}\quad\mbox{for all $\xx\in\Sigma$},
\end{equation}
then
\begin{equation}\label{eqn:slcurv}
\abs{A(\xx)}\leq C_1\epsilon\abs{\xx}\quad\mbox{for all $\xx\in\Sigma\cap B_{\frac{r_1}{2}}(\xx_1)$}.
\end{equation}
\end{lemma}

\begin{proof}
Throughout, let $C$ be a universal constant which may change among lines. Let $r$ be the distance function from $\xx_1$. We define $F=(r_1-r)^2 |A|^2$ on $\Sigma$. Observe that $F$ vanishes on $\partial\Sigma\subset\partial B_{r_1}$. Thus $F$ achieves its finite maximum at a point $\yy$ in $\Sigma$. Let $C_1$ be a constant to be determined and assume that $\epsilon<\min\{1/2,4/C_1\}$. We show that $F(\yy)\geq (C_1\epsilon)^2/16$ leads to a contradiction for $C_1$ sufficiently large.

Define $\sigma$ by $8\sigma |A(\yy)|=C_1\epsilon$. If $F(\yy)\geq (C_1\epsilon)^2/16$, then $2\sigma\leq r_1-r(\yy)$. Thus for $\xx\in\Sigma\cap B_{\sigma}(\yy)$,
\begin{equation}
\abs{r_1-r(\xx)}\geq\abs{r_1-r(\yy)}-\sigma\geq\frac{1}{2}\abs{r_1-r(\yy)},
\end{equation}
and so
\begin{equation}
\frac{1}{4}\left(r_1-r(\yy)\right)^2\abs{A(\xx)}^2 \leq F(\xx) \leq F(\yy),
\end{equation}
implying $|A(\xx)|\leq 2 |A(\yy)|$.

Rescaling $\Sigma$ about $\yy$, we define $\tilde{\Sigma}=\sigma^{-1}(\Sigma-\yy)$. Then by our assumptions,
\begin{equation}
\abs{A(\oo)}=\frac{C_1\epsilon}{8}\leq\frac{1}{2},
\end{equation}
and $|A|\leq1$ on $\tilde{\Sigma}\cap B_1$. Note that $|\xx_1|>1/\epsilon>1$ and $\yy\in B_{r_1}(\xx_1)$. Thus $\yy\neq\oo$ and $\vv=\yy/|\yy|$ is well-defined. Then by \cite[Theorem 5.2]{CM2},
\begin{equation}
\lap\left\la\vv,\nn\right\ra-\frac{1}{2}\left\la\sigma\yy+\sigma^2\xx,\grad\left\la\vv,\nn\right\ra\right\ra+\abs{A}^2\left\la\vv,\nn\right\ra=0 \quad\mbox{on $\tilde{\Sigma}\cap B_1$}.
\end{equation}
Thus by the Schauder interior estimate \cite[Theorem 6.2]{GT},
\begin{equation}\label{eqn:splitcurv}
\sup_{\tilde{\Sigma}\cap B_{\frac{1}{2}}}\abs{\grad\left\la\vv,\nn\right\ra}\leq C\sup_{\tilde{\Sigma}\cap B_1}\abs{\left\la\vv,\nn\right\ra}.
\end{equation}
Note that \eqref{eqn:smallmc} together with $|\xx_1|>1$ implies that $|H|\leq\epsilon$ and $|\la\vv,\nn\ra|\leq 4\epsilon$ on $\tilde{\Sigma}\cap B_1$. Hence it follows from \eqref{eqn:splitcurv} that 
\begin{equation}
\sup_{\tilde{\Sigma}\cap B_{\frac{1}{2}}}\abs{A}\leq C\epsilon. 
\end{equation}
In particular, 
\begin{equation}
\frac{C_1\epsilon}{8}=\abs{A(\oo)}\leq C\epsilon,
\end{equation}
so, if $C_1>8C$, this gives a contradiction. 

Therefore, choosing $C_1=9C$, if $\epsilon<\min\{1/2,4/C_1\}$, then $F<(C_1\epsilon)^2/16$ on $\Sigma$, which implies \eqref{eqn:slcurv} immediately.
\end{proof}

Consequently $\Sigma$ in Lemma \ref{lem:slcurv} contains the graph of a $C^1$ function on the square centered at $\xx_1$ with length ${\sim}1/(\epsilon |\xx_1|)$ and with sides parallel to $\xx_1$ and $\nn(\xx_1)\times\xx_1$. Roughly speaking, the next lemma shows that one can patch these graphs along the radial direction to deduce the existence of a larger graphical region. 

Let $d_\Sigma$ denote the geodesic distance in $\Sigma$. Given $\vv\in\RR^3\setminus\{\oo\}$ and real numbers $a<b$, we define $I_{\vv}(a,b)$ to be the open interval $(a,b)$ on the straight line through $\oo$ oriented by $\vv$. Let $D_r(\xx)\subset\RR^2$ be the open disk centered at $\xx$ with radius $r$.

\begin{lemma}\label{lem:radext}
There exist $\epsilon_2,\eta_0\in (0,1)$ and $C_2>0$ such that for $0<\epsilon<\epsilon_2$ and $\xx_1\in\RR^3$ with $\epsilon |\xx_1|>4$, if $\Sigma$ is a properly embedded self-shrinker in $B_{r_1}(\xx_1)$ with $\xx_1\in\Sigma$ and $\partial\Sigma\subset\partial B_{r_1}(\xx_1)$ for $r_1=1/(\epsilon^2|\xx_1|)$ which satisfies that
\begin{equation}\label{eqn:split}
\abs{H(\xx)}<\epsilon_2\abs{\xx}\quad\mbox{for all $\xx\in\Sigma$},\quad\mbox{and}
\end{equation}
\begin{equation}\label{eqn:oscnormal}
\sup_{\substack{\xx,\yy\in\Sigma \\ d_\Sigma (\xx,\yy)<\epsilon r_1}}\abs{\nn(\xx)-\nn(\yy)}<\epsilon_2,
\end{equation}
then $\Sigma$ contains the graph of a function $u$ on the rectangle
\begin{equation} 
\eta_0r_1\left\{I_{\xx_1}(-1,1)\times I_{\vv_1}(-\epsilon,\epsilon)\right\}+\xx_1
\end{equation}
with $u(\xx_1)=0$ and $|\grad u|\leq C_2$, where $\vv_1=\nn(\xx_1)\times\xx_1$.
\end{lemma}

\begin{proof}
Throughout, let $\epsilon_2$ be a sufficiently small constant to be determined. By the gradient estimate of Ecker-Huisken \cite[Theorem 2.3]{EH2} (see also \cite{CM1}) and the height estimate \cite[Lemma 3]{CM1}, there is a universal constant $C_{EH}>1$ such that, if a mean curvature flow $\{\Sigma_t\}_{t\in [0,1]}$ can be written as ${\rm graph}(w)$, where $w: D_3(\oo)\times [0,1]\to\RR$ with $|\grad w(\cdot,0)|<1$, then $|\grad w(\oo,t)|\leq C_{EH}$ for $0\leq t\leq 1$. 

By \eqref{eqn:split}, we can choose the unit normal $\NN_1$ to the plane spanned by $\xx_1$ and $\vv_1$ such that 
\begin{equation}
\left\la\NN_1,\nn(\xx_1)\right\ra>\sqrt{1-4\epsilon_2^2}.
\end{equation}
Furthermore, it follows from \eqref{eqn:oscnormal} that for $\xx\in\Sigma$ with $d_\Sigma(\xx,\xx_1)<\epsilon r_1$,
\begin{equation}
\left\la\NN_1,\nn(\xx)\right\ra>\sqrt{1-4\epsilon_2^2}-\epsilon_2.
\end{equation}
Thus, if $\epsilon_2$ is small such that $\sqrt{1-4\epsilon_2^2}-\epsilon_2>1/\sqrt{2}$, then $\Sigma$ contains the graph of a function $u$ on 
\begin{equation}
\Omega=\frac{\epsilon r_1}{2}\left\{I_{\xx_1}(-1,1)\times I_{\vv_1}(-1,1)\right\}+\xx_1
\end{equation}
with $u(\xx_1)=0$ and $|\grad u|<1$. Here the gradient bound of $u$ can be deduced from
\begin{equation}
1+\abs{\grad u}^2=\left\la\NN_1,\nn\right\ra^{-2}<2.
\end{equation}

Next we show that $u$ can be extended along the direction of $\xx_1$ with control of its gradient such that $\Sigma$ contains a larger graphical region. Define $\eta\in (0,1)$ by $2(1+4C_{EH}^2)\eta^2=1$. Suppose that $l$ is the maximum number such that there exists a function $\tilde{u}$ defined on
\begin{equation}
\tilde{\Omega}=r_1\left\{I_{\xx_1}(-\eta\epsilon,l^2)\times I_{\vv_1}(-\eta\epsilon,\eta\epsilon)\right\}+\xx_1
\end{equation}
such that ${\rm graph}(\tilde{u})\subset\Sigma$, $\tilde{u}(\xx_1)=0$ and $|\grad\tilde{u}|\leq 2C_{EH}$. Clearly $l$ exists and $2l^2>\epsilon$. And $\tilde{u}=u$ on $\Omega\cap\tilde{\Omega}$ as $\Sigma$ is embedded. We will show that the smallness of $l$ leads to a contradiction for $\epsilon_2$ sufficiently small. 

If $\epsilon_2<10^{-2}$ and $l<\eta/12$, we can define a function $\tilde{v}$ by
\begin{equation}
\tilde{v}(\pp,t)=\sqrt{1-t}\, \tilde{u}\left(\frac{\pp}{\sqrt{1-t}}\right)
\end{equation}
on the space-time domain
\begin{equation}
\left\{6l\epsilon r_1\left(I_{\xx_1}(-1,1)\times I_{\vv_1}(-1,1)\right)+\xx_1\right\}\times [0,T],
\end{equation}
where
\begin{equation}\label{eqn:time}
0<T=1-\left(\frac{\abs{\xx_1}+6l\epsilon r_1}{\abs{\xx_1}+l^2r_1}\right)^2<\left(2l\epsilon r_1\right)^2<10^{-2}.
\end{equation}
Here we use the assumption that $\epsilon |\xx_1|>4$ in \eqref{eqn:time}. Since ${\rm graph}(\tilde{u})\subset\Sigma$ is a self-shrinker, ${\rm graph}(\tilde{v}(\cdot,t))$ moves by mean curvature. 

Let
\begin{equation}
\pp_1=\frac{\xx_1}{\sqrt{1-T}}\quad\mbox{and}\quad\yy_1=\pp_1+\tilde{u}(\pp_1)\NN_1.
\end{equation}
Then the height estimate \cite[Lemma 3]{CM1} gives
\begin{equation}\label{eqn:height}
\abs{\tilde{u}(\pp_1)}=\frac{\abs{\tilde{v}(\xx_1,T)}}{\sqrt{1-T}}<10^2l\epsilon r_1.
\end{equation} 
Also, observe that $|\grad\tilde{v}(\pp,0)|=|\grad\tilde{u}(\pp)|<1$. Thus it follows from the gradient estimate \cite[Theorem 2.3]{EH2} and the definition of $C_{EH}$ that 
\begin{equation}\label{eqn:gradient}
\abs{\grad\tilde{u}(\pp_1)}=\abs{\grad\tilde{v}(\xx_1,T)}\leq C_{EH}. 
\end{equation}
Hence, combining \eqref{eqn:oscnormal}, \eqref{eqn:height} and \eqref{eqn:gradient}, there exists an $\epsilon_2>0$ small depending only on $C_{EH}$ such that $B_{\epsilon r_1}(\yy_1)\subset B_{r_1}(\xx_1)$ and for $\xx\in\Sigma$ with $d_\Sigma(\xx,\yy_1)<\epsilon r_1$,
\begin{equation}\label{eqn:normal}
\left\la\NN_1,\nn(\xx)\right\ra\geq\frac{1}{\sqrt{1+C_{EH}^2}}-\epsilon_2>\frac{1}{\sqrt{1+4C_{EH}^2}}.
\end{equation}

Note that 
\begin{equation}\label{eqn:radial}
\abs{\pp_1}=\abs{\xx_1}+l^2r_1-\frac{6l\epsilon r_1}{\sqrt{1-T}}>\abs{\xx_1}+l^2r_1-12l\epsilon r_1.
\end{equation}
Hence, if $l<\eta/24$, it follows from \eqref{eqn:normal} and \eqref{eqn:radial} that $\tilde{u}$ can be extended to  
\begin{equation}
r_1\left\{I_{\xx_1}\left(-\eta\epsilon,l^2+\frac{\eta\epsilon}{2}\right)\times I_{\vv_1}(-\eta\epsilon,\eta\epsilon)\right\}+\xx_1,
\end{equation}
satisfying that ${\rm graph}(\tilde{u})\subset\Sigma$ and $|\grad\tilde{u}|\leq 2C_{EH}$. This contradicts the maximality of $l$. Therefore there is an $\epsilon_2>0$ small depending only on $C_{EH}$ such that $l\geq\eta/24$.

Similarly one can extend ${\rm graph}(u)$ along the opposite direction of $\xx_1$ to a larger graphical region in $\Sigma$. Namely, we define $\hat{\eta}\in (0,1)$ by $2(1+16C_{EH}^2)\hat{\eta}^2=1$. Suppose that $\hat{l}$ is the maximum number such that there exists a function $\hat{u}$ on
\begin{equation}
\hat{\Omega}=r_1\left\{I_{\xx_1}(-\hat{l}^2,\hat{\eta}\epsilon)\times I_{\vv_1}(-\hat{\eta}\epsilon,\hat{\eta}\epsilon)\right\}+\xx_1
\end{equation}
such that ${\rm graph}(\hat{u})\subset\Sigma$, $\hat{u}(\xx_1)=0$ and $|\grad\hat{u}|\leq 4C_{EH}$. Clearly $\hat{l}$ exists and $2\hat{l}^2>\epsilon$. And $\hat{u}=u$ on $\Omega\cap\hat{\Omega}$. We will bound $\hat{l}$ from below by contradiction.

First we derive a $L^\infty$ bound of $\hat{u}$ on the ray from $\oo$ to $\xx_1$ by invoking that $\{\sqrt{1-t}\, \Sigma\}_{t<1}$ is a mean curvature flow. Assume that $\epsilon_2<10^{-2}$ and $\hat{l}<\hat{\eta}/12$. We can define
\begin{equation}
\hat{v}(\pp,t)=\sqrt{1-t}\, \hat{u}\left(\frac{\pp}{\sqrt{1-t}}\right)
\end{equation}
on the space-time domain
\begin{equation}
\left\{6\hat{l}\epsilon r_1\left(I_{\xx_1}(-1,1)\times I_{\vv_1}(-1,1)\right)+\hat{\pp}_1\right\}\times [0,\hat{T}],
\end{equation}
where 
\begin{equation}
\hat{\pp}_1=\left(\abs{\xx_1}-\hat{l}^2r_1+6\hat{l}\epsilon r_1\right)\frac{\xx_1}{\abs{\xx_1}},\quad 0<\hat{T}=1-\frac{\abs{\hat{\pp}_1}^2}{\abs{\xx_1}^2}<\left(2\hat{l}\epsilon r_1\right)^2<10^{-2}.
\end{equation}
Here we use the assumption that $\epsilon |\xx_1|>4$ to estimate $\hat{T}$. And ${\rm graph}(\hat{v}(\cdot,t))$ moves by mean curvature, as ${\rm graph}(\hat{u})\subset\Sigma$ is a self-shrinker. Observe that 
\begin{equation}
\hat{v}(\hat{\pp}_1,\hat{T})=\left(1-\hat{T}\right)^{\frac{1}{2}}\hat{u}\left(\xx_1\right)=0 \quad\mbox{and}\quad\abs{\hat{\pp}_1}>\frac{\abs{\xx_1}}{2}. 
\end{equation}
Thus the height estimate \cite[Lemma 3]{CM1} implies that, if $|\hat{\pp}_1|\leq p\leq |\xx_1|$, then
\begin{equation}\label{eqn:height1}
\abs{\hat{u}\left(\frac{p\xx_1}{\abs{\xx_1}}\right)}=\frac{p}{\abs{\hat{\pp}_1}}\abs{\hat{v}\left(\hat{\pp}_1,1-\frac{\abs{\hat{\pp}_1}^2}{p^2}\right)}<C\hat{l}\epsilon r_1
\end{equation}
for some $C$ depending only on $C_{EH}$. 

Next let $\hat{\yy}_1=\hat{\pp}_1+\hat{u}(\hat{\pp}_1)\NN_1$. By \eqref{eqn:height1} and the assumption that $\epsilon |\xx_1|>4$, we get that
\begin{equation}
\abs{\frac{\hat{\yy}_1}{\abs{\hat{\yy}_1}}-\frac{\xx_1}{\abs{\xx_1}}}=\abs{\frac{\hat{\yy}_1}{\abs{\hat{\yy}_1}}-\frac{\hat{\pp}_1}{\abs{\hat{\pp}_1}}}\leq\frac{2\abs{\hat{u}(\hat{\pp}_1)}}{\abs{\hat{\pp}_1}}<C^\prime \epsilon_2
\end{equation}
for some $C^\prime$ depending only on $C$. Together with \eqref{eqn:split}, this further gives that for $\epsilon_2$ sufficiently small,
\begin{equation}\label{eqn:split1}
\abs{\left\la\xx_1,\nn(\hat{\yy}_1)\right\ra}<\left(C^\prime+2\right)\epsilon_2\abs{\xx_1}<\frac{\abs{\xx_1}}{2}.
\end{equation}
Thus by \eqref{eqn:oscnormal}, \eqref{eqn:height1} and \eqref{eqn:split1}, there exists an $\epsilon_2>0$ depending only on $C^\prime$ such that $B_{\epsilon r_1}(\hat{\yy}_1)\subset B_{r_1}(\xx_1)$ and for $\xx\in\Sigma$ with $d_\Sigma (\xx,\hat{\yy}_1)<\epsilon r_1$,
\begin{equation}
\left\la\hat{\NN}_1,\nn(\xx)\right\ra>\sqrt{1-\left(C^\prime+2\right)^2\epsilon_2^2}-\epsilon_2>\frac{1}{\sqrt{2}},
\end{equation}
where $\hat{\NN}_1$ is the unit normal to the plane spanned by $\xx_1$ and $\hat{\vv}_1=\nn (\hat{\yy}_1)\times\xx_1\neq\oo$. Furthermore, if $\hat{l}<\eta/(4C)$, then $|\la\hat{\vv}_1,\hat{\yy}_1\ra|<\eta\epsilon r_1|\hat{\vv}_1|/4$ and the geodesic ball in $\Sigma$ centered at $\hat{\yy}_1$ with radius $\epsilon r_1$ contains the graph of a function $u^\prime$ defined on
\begin{equation}
\eta\epsilon r_1\left\{I_{\xx_1}(-1,1)\times I_{\hat{\vv}_1}(-1,1)\right\}+\hat{\pp}_1
\end{equation}
with $\hat{\yy}_1\in {\rm graph}(u^\prime)$ and $|\grad u^\prime|<1$. Apply the previous arguments\footnote{Note that $u^\prime (\hat{\pp}_1)$ is bounded by ${\sim}\epsilon r_1$, which, together with $|\grad u^\prime|<1$, is the only assumption needed previously.}, which we use to extend $u$ along the direction of $\xx_1$, to $u^\prime$. Hence, for $\epsilon_2$ sufficiently small, there exists a function $\hat{u}^\prime$, which is an extension of $u^\prime$, defined on 
\begin{equation}
\hat{\Omega}^\prime=r_1\left\{I_{\xx_1}(-\eta\epsilon,10^{-4}\eta^2)\times I_{\hat{\vv}_1}(-\eta\epsilon,\eta\epsilon)\right\}+\hat{\pp}_1
\end{equation}
such that ${\rm graph}(\hat{u}^\prime)\subset\Sigma$ with $\hat{\yy}_1\in {\rm graph}(\hat{u}^\prime)$ and $|\grad\hat{u}^\prime|\leq 2C_{EH}$.

Moreover, we show that if $\hat{l}$ is small, then $\hat{u}^\prime (\xx_1)=0$, i.e., $\xx_1\in {\rm graph}(\hat{u}^\prime)$. To see this, observe that, if $\hat{\pp}\in\hat{\Omega}^\prime$ satisfies $|\la\hat{\vv}_1,\hat{\pp}\ra|<\eta\epsilon r_1|\hat{\vv}_1|/4$, the geodesic ball in $\Sigma$ centered at $\hat{\pp}+\hat{u}^\prime (\hat{\pp})\hat{\NN}_1$ with radius $\eta\epsilon r_1/4$ is contained in ${\rm graph}(\hat{u}^\prime)$. Consider the curve
\begin{equation}
\gamma: [\abs{\hat{\pp}_1},\abs{\xx_1}]\to\RR^3,\quad
\gamma (p)= \frac{p\xx_1}{\abs{\xx_1}}+\hat{u}\left(\frac{p\xx_1}{\abs{\xx_1}}\right)\NN_1.
\end{equation}
If $\hat{l}<\min\{\eta/(4C),10^{-2}\eta\}$, by \eqref{eqn:height1}, we have that $|\la\gamma (p),\hat{\vv}_1\ra|<\eta\epsilon r_1|\hat{\vv}_1|/4$ and by a covering argument, $\gamma\subset{\rm graph}(\hat{u}^\prime)$ as $\hat{\yy}_1\in {\rm graph}(\hat{u})\cap {\rm graph}(\hat{u}^\prime)$ and $\xx_1\in\hat{\Omega}^\prime$. Thus $\xx_1=\gamma(|\xx_1|)\in {\rm graph}(\hat{u}^\prime)$.
 
Finally, by \eqref{eqn:split}, \eqref{eqn:split1} and the gradient bound of $\hat{u}^\prime$, we get that
\begin{equation}
\begin{split}
& \left\la\NN_1,\nn(\hat{\yy}_1)\right\ra\geq\left\la\hat{\NN}_1,\nn(\xx_1)\right\ra-\abs{\NN_1-\nn(\xx_1)}-\abs{\hat{\NN}_1-\nn(\hat{\yy}_1)}\\
& \quad\quad\, \geq\frac{1}{\sqrt{1+4C_{EH}^2}}-\left(2-2\sqrt{1-4\epsilon_2^2}\right)^{\frac{1}{2}}-\left(2-2\sqrt{1-\left(C^\prime+2\right)^2\epsilon_2^2}\right)^{\frac{1}{2}}.
\end{split}
\end{equation}
Hence there exists an $\epsilon_2>0$ depending only on $C_{EH}$ and $C^\prime$ such that 
\begin{equation}
\left\la\NN_1,\nn(\hat{\yy}_1)\right\ra\geq\frac{1}{\sqrt{1+9C_{EH}^2}}\, ,\quad\mbox{i.e.,}\quad\abs{\grad\hat{u}(\hat{\pp}_1)}\leq 3C_{EH}, 
\end{equation}
and invoking that $B_{\epsilon r_1}(\hat{\yy}_1)\subset B_{r_1}(\xx_1)$ and \eqref{eqn:oscnormal} again, $\hat{u}$ can be extended to 
\begin{equation}
r_1\left\{I_{\xx_1}\left(-\hat{l}^2-\frac{\hat{\eta}\epsilon}{2},\hat{\eta}\epsilon\right)\times I_{\vv_1}(-\hat{\eta}\epsilon,\hat{\eta}\epsilon)\right\}+\xx_1
\end{equation}
with ${\rm graph}(\hat{u})\subset\Sigma$ and $|\grad\hat{u}|\leq 4C_{EH}$. This contradicts the maximality of $\hat{l}$ and thus $\hat{l}\geq\min\{\hat{\eta}/12,\eta/(4C),10^{-2}\eta\}$ for $\epsilon_2$ sufficiently small. 

Therefore, choosing $\epsilon_2$ sufficiently small, $\eta_0=\min\{\eta,\hat{\eta},l^2,\hat{l}^2\}$ and $C_2=4C_{EH}$, the lemma follows immediately from that $u=\tilde{u}=\hat{u}$ on $\Omega\cap\tilde{\Omega}\cap\hat{\Omega}$.
\end{proof}

Furthermore, in the following lemma, we estimate the rate of change of the unit normal for graphical self-shrinkers.

\begin{lemma}\label{lem:imposcnormal}
There exists a $C_3>0$ depending only on $C_2$ such that for $\epsilon\in (0,1/4)$ and $p\in\RR$ with $p\epsilon>4$, if $\Sigma$ is a self-shrinker given by the graph of a function 
\begin{equation}
u(p_1,p_2): (p-2r,p+2r)\times (-4\epsilon r,4\epsilon r)\to\RR
\end{equation}
with $u(p,0)=0$ and $|\grad u|<C_2$ for $r=1/(p\epsilon^2)$, then on $(p-r,p+r)\times (-\epsilon r,\epsilon r)$,
\begin{equation}
\abs{\partial_1u}\leq C_3\epsilon\quad\mbox{and}\quad\epsilon\abs{\partial_2^2u}+\abs{\partial_1^2u}+\abs{\partial_1\partial_2u}\leq\frac{C_3}{r},
\end{equation}
where $\partial_i$ denotes taking partial derivative with respect to $p_i$.
\end{lemma}

\begin{proof}
As $0<\epsilon<1/4$ and $p\epsilon>4$, we can define
\begin{equation}
v: (p-2r,p-2r+4\epsilon r)\times (-2\epsilon r,2\epsilon r)\times [0,T]\to\RR
\end{equation}
by
\begin{equation}
v(p_1,p_2,t)=\sqrt{1-t}\, u\left(\frac{p_1}{\sqrt{1-t}},\frac{p_2}{\sqrt{1-t}}\right),
\end{equation}
where 
\begin{equation}
0<T=1-\left(\frac{p-2r+2\epsilon r}{p+r}\right)^2<6\epsilon^2r^2<\frac{3}{8}.
\end{equation} 
Since ${\rm graph}(u)=\Sigma$ is a self-shrinker, ${\rm graph}(v(\cdot,t))$ moves by mean curvature.

By our assumptions, we have that 
\begin{equation}
T^\prime=1-\left(\frac{p-2r+2\epsilon r}{p-r}\right)^2\geq\frac{\epsilon^2r^2}{2}.
\end{equation}
Note that $v(p,0,t_0)=u(p,0)=0$ for some $T^\prime<t_0<T$ and
\begin{equation}
\abs{\grad v(p_1,p_2,t)}=\abs{\grad u \left(\frac{p_1}{\sqrt{1-t}},\frac{p_2}{\sqrt{1-t}}\right)}<C_2.
\end{equation}
Thus the height estimate \cite[Lemma 3]{CM1} and the higher regularity \cite[Theorem 3.4]{EH2} imply that for some $C$ depending only on $C_2$,
\begin{equation}\label{eqn:highder}
\abs{v}\leq C\epsilon r\quad\mbox{and}\quad\abs{\partial_t v}+\sum_{i,j=1}^2\abs{\partial_i\partial_j v}+\epsilon r\sum_{i=1}^2\abs{\partial_t\partial_i v}\leq\frac{C}{\epsilon r}
\end{equation}
on 
\begin{equation}
(p-2r+\epsilon r,p-2r+3\epsilon r)\times (-\epsilon r,\epsilon r)\times [T^\prime,T].
\end{equation}

Observe that given $|p_1-p|<r$ and $|p_2|<\epsilon r$, there exists $t\in (T^\prime,T)$ such that 
\begin{equation}
\left(\sqrt{1-t}\, p_1,\sqrt{1-t}\, p_2\right)\in (p-2r+\epsilon r,p-2r+3\epsilon r)\times (-\epsilon r,\epsilon r).
\end{equation}
And by the definition of $v$ and the chain rule of taking derivatives,
\begin{equation}
\begin{split}
\partial_t v(p_1,p_2,t) & =\frac{-v(p_1,p_2,t)}{2\left(1-t\right)}+\frac{p_1}{2\left(1-t\right)}\, \partial_1 u\left(\frac{p_1}{\sqrt{1-t}},\frac{p_2}{\sqrt{1-t}}\right)\\
& \quad +\frac{p_2}{2\left(1-t\right)}\, \partial_2 u\left(\frac{p_1}{\sqrt{1-t}},\frac{p_2}{\sqrt{1-t}}\right).
\end{split}
\end{equation}
Similarly,
\begin{align}
\partial_i\partial_j v(p_1,p_2,t) & =\frac{1}{\sqrt{1-t}}\, \partial_i\partial_j u\left(\frac{p_1}{\sqrt{1-t}},\frac{p_2}{\sqrt{1-t}}\right),\quad\mbox{and}\\
\begin{split}
\partial_t\partial_i v(p_1,p_2,t) & =\frac{p_1}{2\left(1-t\right)^{\frac{3}{2}}}\, \partial_1\partial_i u\left(\frac{p_1}{\sqrt{1-t}},\frac{p_2}{\sqrt{1-t}}\right)\\
&\quad+\frac{p_2}{2\left(1-t\right)^{\frac{3}{2}}}\, \partial_2\partial_i u\left(\frac{p_1}{\sqrt{1-t}},\frac{p_2}{\sqrt{1-t}}\right).
\end{split}
\end{align}
Therefore the lemma follows immediately from \eqref{eqn:highder}.
\end{proof}

Finally, combining Lemmas \ref{lem:radext} and \ref{lem:imposcnormal} with the total curvature bound, we establish the bootstrap machinery which is the key to the proof of Theorem \ref{thm:sheeting}. 

\begin{lemma}\label{lem:iteration}
There exists an $\epsilon_3\in (0,1)$ such that given $\kappa>0$ and $\delta\in (0,1)$, there exists an $\eta_1\in (0,1)$ such that for $0<\epsilon<\epsilon_3$ and $\xx_1\in\RR^3$ with $\epsilon^2|\xx_1|>1$, if $\Sigma$ is a properly embedded self-shrinker in $B_{r_1}(\xx_1)$ with $\xx_1\in\Sigma$ and $\partial\Sigma\subset\partial B_{r_1}(\xx_1)$ for $r_1=1/(\epsilon^2|\xx_1|)$, and if 
\begin{equation}\label{eqn:split2}
\abs{H(\xx)}<\epsilon_3\abs{\xx}\quad\mbox{for all $\xx\in\Sigma$},
\end{equation}
\begin{equation}\label{eqn:oscnormal1}
\sup_{\substack{\xx,\yy\in\Sigma \\ d_\Sigma(\xx,\yy)<\epsilon r_1}}\abs{\nn(\xx)-\nn(\yy)}<\delta\epsilon_3,\quad\mbox{and}
\end{equation}
\begin{equation}\label{eqn:totalcurv1}
\int_\Sigma\abs{A}^2d\mathcal{H}^2<\kappa,
\end{equation}
then 
\begin{equation}
\sup_{\substack{\xx\in\Sigma \\ d_\Sigma(\xx,\xx_1)<\eta_1r_1}}\abs{\nn(\xx)-\nn(\xx_1)}<\delta\epsilon_3.
\end{equation}
\end{lemma}

\begin{proof}
First we claim that for $\epsilon_3$ sufficiently small, there exists an $\eta\in (0,1)$ depending only on $C_2$ such that given $\yy\in\Sigma\cap B_{r_1/2}(\xx_1)$ with $|\nn(\yy)-\nn(\xx_1)|<\delta\epsilon_3$, the self-shrinker $\Sigma$ contains the graph of a function $u$ defined on 
\begin{equation}
\eta\delta\epsilon_3r_1\left\{I_{\xx_1}(-1,1)\times I_{\vv_1}(-\epsilon,\epsilon)\right\}+\yy,
\end{equation}
where $\vv_1=\nn (\xx_1)\times\xx_1$, satisfying that $u(\yy)=0$ and 
\begin{equation}
\sup_{\xx\in {\rm graph}(u)}\abs{\nn(\xx)-\nn(\yy)}<\frac{\delta\epsilon_3}{8}.
\end{equation}

For $\yy\in\Sigma\cap B_{r_1/2}(\xx_1)$, as $\epsilon<1$ and $\epsilon^2|\xx_1|>1$, it is easy to see that $2|\yy|>|\xx_1|$ and $B_{\rho}(\yy)\subset B_{r_1}(\xx_1)$ for $\rho=1/(4\epsilon^2|\yy|)$. We choose $\epsilon_3<\min\{1/4,\epsilon_2/2,\eta_0/32\}$ so that Lemmas \ref{lem:radext} and \ref{lem:imposcnormal} are applicable to $\Sigma\cap B_{\rho}(\yy)$. Thus there exist constants $\tilde{\eta}\in (0,1)$ and $C>0$, depending on $\eta_0$, $C_2$ and $C_3$, and a function $\tilde{u}$ defined on 
\begin{equation}
\tilde{\eta}\delta\epsilon_3r_1\left\{I_{\yy}(-1,1)\times I_{\vv}(-\epsilon,\epsilon)\right\}+\yy,
\end{equation}
where $\vv=\nn (\yy)\times\yy$, such that ${\rm graph}(\tilde{u})\subset\Sigma$, $\tilde{u}(\yy)=0$, 
\begin{equation}\label{eqn:height2}
\sup_{\xx\in {\rm graph}(\tilde{u})}\abs{\tilde{u}(\xx)}<C\delta\epsilon_3\epsilon r_1, \quad\mbox{and}
\end{equation} 
\begin{equation}\label{eqn:oscnormal2}
\sup_{\xx\in {\rm graph}(\tilde{u})}\abs{\nn(\xx)-\nn(\yy)}<\frac{\delta\epsilon_3}{8}.
\end{equation}

To conclude the proof of the claim, we need to rewrite ${\rm graph}(\tilde{u})$ as the graph over the plane $\PP_1$ through $\yy$ spanned by $\xx_1$ and $\vv_1$. Consider the curve 
\begin{equation}
\gamma: (-\tilde{\eta}\delta\epsilon_3r_1,\tilde{\eta}\delta\epsilon_3r_1)\to\RR^3, \quad\gamma(p)=\yy+\frac{p\yy}{\abs{\yy}}+\tilde{u}\left(\yy+\frac{p\yy}{\abs{\yy}}\right)\NN,
\end{equation}
where $\NN$ is the unit normal to the plane spanned by $\yy$ and $\vv$ so that $\la\NN,\nn(\yy)\ra>0$. First we show that the projection of $\gamma$ onto $\PP_1$ stays close to the axis parallel to $\xx_1$. Note that
\begin{equation}
\abs{\frac{\yy}{\abs{\yy}}-\frac{\xx_1}{\abs{\xx_1}}}<2\epsilon^2r_1^2<2\epsilon^2,
\end{equation}
as $\yy\in B_{r_1/2}(\xx_1)$ and $r_1<1$, and \eqref{eqn:split2} gives
\begin{equation}\label{eqn:diffnormal}
\begin{split}
\abs{\NN-\NN_1} & \leq\abs{\NN-\nn(\yy)}+\abs{\nn(\yy)-\nn(\xx_1)}+\abs{\NN_1-\nn(\xx_1)}\\
& <2\left(2-2\sqrt{1-4\epsilon_3^2}\right)^{\frac{1}{2}}+\delta\epsilon_3,
\end{split}
\end{equation}
where $\NN_1$ is the unit normal to $\PP_1$. Thus, together with \eqref{eqn:height2}, it follows that
\begin{equation}\label{eqn:width}
\abs{\left\la\frac{\vv_1}{\abs{\vv_1}},\gamma(p)-\yy\right\ra}\leq\abs{p}\abs{\frac{\yy}{\abs{\yy}}-\frac{\xx_1}{\abs{\xx_1}}}+\sup\abs{\tilde{u}}\abs{\NN-\NN_1}<C^\prime\delta\epsilon_3^2\epsilon r_1
\end{equation}
for some $C^\prime$ depending only on $C$. Similarly,
\begin{equation}\label{eqn:length}
\abs{\left\la\frac{\xx_1}{\abs{\xx_1}},\gamma(p)-\yy\right\ra-p}<C^\prime\delta\epsilon_3^2\epsilon r_1.
\end{equation}
Next it follows from \eqref{eqn:split}, \eqref{eqn:oscnormal2} and \eqref{eqn:diffnormal} that for $\xx\in {\rm graph}(\tilde{u})$,
\begin{equation}\label{eqn:gradient1}
\begin{split}
\left\la\NN_1,\nn(\xx)\right\ra & \geq\left\la\NN,\nn(\yy)\right\ra-\abs{\nn(\xx)-\nn(\yy)}-\abs{\NN-\NN_1}\\
& >\sqrt{1-4\epsilon_3^2}-2\left(2-2\sqrt{1-4\epsilon_3^2}\right)^{\frac{1}{2}}-2\delta\epsilon_3.
\end{split}
\end{equation}
Note that ${\rm graph}(\tilde{u})$ contains the tubular neighborhood of $\gamma$ in $\Sigma$ with radius $\tilde{\eta}\delta\epsilon_3\epsilon r_1$. Hence, combining \eqref{eqn:width}, \eqref{eqn:length} and \eqref{eqn:gradient1}, there exists an $\epsilon_3>0$ depending only on $\tilde{\eta}$ and $C^\prime$ such that ${\rm graph}(\tilde{u})$ contains the graph of a function $u$ on
\begin{equation}
\frac{1}{2}\tilde{\eta}\delta\epsilon_3r_1\left\{I_{\xx_1}(-1,1)\times I_{\vv_1}(-\epsilon,\epsilon)\right\}+\yy,
\end{equation}
proving the claim.

Now let $l>0$ be the maximal number such that there exists a function $u^\prime$ on 
\begin{equation}
\Omega=\eta\delta\epsilon_3r_1\left\{I_{\xx_1}(-1,1)\times I_{\vv_1}(-l,l)\right\}+\xx_1
\end{equation}
satisfying that ${\rm graph}(u^\prime)\subset\Sigma$, $u^\prime(\xx_1)=0$ and 
\begin{equation}
\sup_{\xx\in {\rm graph}(u^\prime)}\abs{\nn(\xx)-\nn(\xx_1)}<\delta\epsilon_3. 
\end{equation}
Then the claim implies that $l\geq\epsilon$ and by \eqref{eqn:split2}, $|\grad u^\prime|<1$ for $\epsilon_3$ sufficiently small.

Via the function $u^\prime$, we identify $(p_1,p_2)\in\Omega$ with
\begin{equation}
\xx=\frac{p_1\xx_1}{\abs{\xx_1}}+\frac{p_2\vv_1}{\abs{\vv_1}}+u^\prime(p_1,p_2)\NN_1\in {\rm graph}(u^\prime),
\end{equation}
and $\nn(p_1,p_2)=\nn(\xx)$. Let $\alpha=\eta\delta\epsilon_3r_1$ and $l^\prime=l-\epsilon/2$. If $l<1/4$, then $|u^\prime(\abs{\xx_1},\alpha l^\prime)|<r_1/4$ and so $(|\xx_1|,\alpha l^\prime)\in\Sigma\cap B_{r_1/2}(\xx_1)$. Thus by the claim and the maximality of $l$, we may assume that 
\begin{equation}
\abs{\nn(|\xx_1|,\alpha l^\prime)-\nn(|\xx_1|,0)}>\frac{\delta\epsilon_3}{2},
\end{equation}
which further implies that 
\begin{equation}
\abs{\nn(p_1,\alpha l^\prime)-\nn(p_1,0)}>\frac{\delta\epsilon_3}{4},\quad\mbox{if }\abs{p_1-\abs{\xx_1}}<\alpha. 
\end{equation}
Hence we get that
\begin{equation}
\begin{split}
\frac{1}{2}\alpha\delta\epsilon_3< & \int_{\abs{\xx_1}-\alpha}^{\abs{\xx_1}+\alpha}\abs{\nn(p_1,\alpha l^\prime)-\nn(p_1,0)}dp_1\leq\sqrt{2}\int_{{\rm graph}(u^\prime)}\abs{A} d\mathcal{H}^2\\
\leq &\, \sqrt{2}\left(\int_{{\rm graph}(u^\prime)}\abs{A}^2 d\mathcal{H}^2\right)^{\frac{1}{2}}\left(\int_{{\rm graph}(u^\prime)}1\, d\mathcal{H}^2\right)^{\frac{1}{2}}\leq 4\alpha\left(l\kappa\right)^{\frac{1}{2}}.
\end{split}
\end{equation}
Here we use the Cauchy-Schwarz inequality in the penultimate inequality and the total curvature bound \eqref{eqn:totalcurv1} and $|\grad u^\prime|<1$ in the last one. This implies that $l\geq(\delta\epsilon_3)^2/(2^6\kappa)$. 

Therefore, for $\epsilon_3$ sufficiently small, choosing $\eta_1=\min\{\eta\delta\epsilon_3,\eta\delta^2\epsilon_3^3/(2^6\kappa)\}$, the lemma follows from that the geodesic ball in $\Sigma$ centered at $\xx_1$ with radius $\eta_1r_1$ is contained in ${\rm graph}(u^\prime)$.
\end{proof}

Now we present the proof of Theorem \ref{thm:sheeting} by applying Lemma \ref{lem:iteration} repeatedly.

\begin{proof}[Proof of Theorem \ref{thm:sheeting}]
We define $\delta>0$ by 
\begin{equation}\label{eqn:delta}
\left(\delta\epsilon_3\right)^2=2-\frac{2}{\sqrt{1+\delta_0^2}}.
\end{equation}
Since it suffices to prove the theorem for $\delta_0$ small, we may assume that $\delta<1$. Let $\eta_1$ be the constant in Lemma \ref{lem:iteration} for $\kappa=\kappa_0$ and $\delta$ be given by \eqref{eqn:delta}. And let $\epsilon=2C_1\epsilon_0/(\delta\epsilon_3)$ and $K$ be the maximum positive integer such that $4^{K+1}<\epsilon^2|\xx_0|$. Finally we set $\rho_0=1/2$ and if $K>1$,
\begin{equation}
\rho_k=\rho_{k-1}-\frac{2^{2k-3}}{\epsilon^2\abs{\xx_0}}\quad\mbox{for $1\leq k\leq K-1$}.
\end{equation}

Now we choose 
\begin{equation}
\epsilon_0=\min\left\{\epsilon_1,\epsilon_3,\frac{\eta_1\delta\epsilon_3}{16C_1},\frac{\delta\epsilon_3^2}{4C_1}\right\}\quad\mbox{and}\quad R_0=\max\left\{\frac{2}{\epsilon_0},\frac{64}{\epsilon^2}\right\}
\end{equation}
so to apply Lemmas \ref{lem:slcurv} and \ref{lem:iteration}. Then $K>1$ and as $4^{K+1}<\epsilon^2|\xx_0|$, 
\begin{equation}
\rho_{K-1}=\frac{1}{2}-\sum_{k=1}^{K-1}\frac{2^{2k-3}}{\epsilon^2\abs{\xx_0}}>\frac{1}{2}-\frac{2^{2K-3}}{3\epsilon^2\abs{\xx_0}}>\frac{1}{4}.
\end{equation}
We claim that for $0\leq k\leq K-1$, given $\xx_1\in\Sigma\cap B_{\rho_k}(\xx_0)$, 
\begin{equation}\label{eqn:oscnormal3}
\sup_{\substack{\xx\in\Sigma\cap B_{\rho_k}(\xx_0) \\ d_\Sigma(\xx,\xx_1)<2^k/\left(\epsilon\abs{\xx_1}\right)}}\abs{\nn(\xx)-\nn(\xx_1)}<\delta\epsilon_3.
\end{equation}

To show the claim, we use an induction argument. For $k=0$, \eqref{eqn:oscnormal3} follows from our assumption \eqref{eqn:slmc} and Lemma \ref{lem:slcurv}. Inductively, suppose that \eqref{eqn:oscnormal3} holds true for $k-1$. Given $\xx_1\in\Sigma\cap B_{\rho_k}(\xx_0)$, as $|\xx|\leq 2|\yy|$ for any $\xx,\yy\in B_1(\xx_0)$, $B_{r_1}(\xx_1)\subset B_{\rho_{k-1}}(\xx_0)$ for $r_1=4^{k-2}/(\epsilon^2|\xx_1|)$ and
\begin{equation}
\sup_{\substack{\xx,\yy\in\Sigma\cap B_{r_1}(\xx_1) \\ d_\Sigma(\xx,\yy)<2^{2-k}\epsilon r_1}}\abs{\nn(\xx)-\nn(\yy)}<\delta\epsilon_3.
\end{equation}
Together with \eqref{eqn:slmc}, it follows from Lemma \ref{lem:iteration} that 
\begin{equation}
\sup_{\substack{\xx\in\Sigma \\ d_\Sigma(\xx,\xx_1)<\eta_1r_1}}\abs{\nn(\xx)-\nn(\xx_1)}<\delta\epsilon_3.
\end{equation}
Thus \eqref{eqn:oscnormal3} for $k$ follows immediately from that 
\begin{equation}
\eta_1r_1\geq\left(\frac{\eta_1\delta\epsilon_3}{16C_1\epsilon_0}\right)\left(\frac{2^k}{\epsilon\abs{\xx_1}}\right)\geq\frac{2^k}{\epsilon\abs{\xx_1}}.
\end{equation}

Finally let $r_0^\prime=4^{K-2}/(\epsilon^2|\xx_0|)$. Thus by the definition of $K$, we have that $2^{-8}<r_0^\prime<2^{-6}$. Hence, applying Lemma \ref{lem:iteration} to $\Sigma\cap B_{r^\prime_0}(\xx_0)$, it follows from \eqref{eqn:slmc} and \eqref{eqn:oscnormal3} for $k=K-1$ that
\begin{equation}
\sup_{\substack{\xx\in\Sigma \\ d_\Sigma(\xx,\xx_0)<\eta_1r^\prime_0}}\abs{\nn(\xx)-\nn(\xx_0)}<\delta\epsilon_3.
\end{equation}
Therefore, by the definition of $\delta$, it follows that $\Sigma$ contains the graph of a function $u$ on the disk $D_{r_0}(\xx_0)\subset T_{\xx_0}\Sigma$ with $u(\xx_0)=0$ and $|\grad u|<\delta_0$ for $r_0=\eta_1r_0^\prime/2$ and the theorem follows immediately from that $B_{r_0}(\xx_0)\subset D_{r_0}(\xx_0)\times\RR$. 
\end{proof}

\section{Applications of Theorem \ref{thm:sheeting}}
There are several interesting consequences of Theorem \ref{thm:sheeting} on the geometry and asymptotic behavior of $2$-dimensional self-shrinkers; see Theorems \ref{thm:seqsheeting} and \ref{thm:bmc} and Corollary \ref{cor:cylinder} in the introduction. This section is devoted to their proofs.

First, using the monotonicity formula and the local Gauss-Bonnet estimate, we can apply Theorem \ref{thm:sheeting} to prove Theorem \ref{thm:seqsheeting} on the asymptotic behavior of properly embedded self-shrinkers with finite genus.

\begin{proof}[Proof of Theorem \ref{thm:seqsheeting}]
First we show that there is an increasing sequence $\lambda_i\to\infty$ with $\lambda_{i+1}/\lambda_i\to 1$ such that $\Sigma-\lambda_i\vv_0$ asymptotically splits off a line in the $L^2$ sense. To achieve this, we define $\Sigma_s$ by
\begin{equation}
\Sigma_s=\Sigma-\e^{\frac{s}{2}}\vv_0.
\end{equation}
Then the self-shrinker equation \eqref{eqn:shrinker} gives that for $s\geq 0$ and $\xx\in\Sigma_s$,
\begin{equation}\label{eqn:rmcf}
\left\la\partial_s\xx,\nn\right\ra=-H+\frac{1}{2}\left\la\xx,\nn\right\ra.
\end{equation}
Thus
\begin{equation}\label{eqn:monotone}
\frac{d}{ds}\int_{\Sigma_s}\e^{-\frac{\abs{\xx}^2}{4}}d\mathcal{H}^2=-\int_{\Sigma_s}\abs{H-\frac{1}{2}\left\la\xx,\nn\right\ra}^2\e^{-\frac{\abs{\xx}^2}{4}}d\mathcal{H}^2.
\end{equation}
Integrating \eqref{eqn:monotone}, it follows from \cite[Theorem 1.1]{DX} that
\begin{equation}
\int_0^\infty\int_{\Sigma_s}\abs{H-\frac{1}{2}\left\la\xx,\nn\right\ra}^2\e^{-\frac{\abs{\xx}^2}{4}}d\mathcal{H}^2ds<\int_{\Sigma_0}\e^{-\frac{\abs{\xx}^2}{4}}d\mathcal{H}^2<\infty.
\end{equation}
Hence there exists an increasing sequence $s_i$ of positive numbers such that 
\begin{equation}
\int_{s_i}^\infty\int_{\Sigma_s}\abs{H-\frac{1}{2}\left\la\xx,\nn\right\ra}^2\e^{-\frac{\abs{\xx}^2}{4}}d\mathcal{H}^2ds<\frac{1}{i^2}.
\end{equation} 
Therefore, for each $i$, there exists a sequence $s_{i,j}$ of numbers in $[s_i,s_{i+1})$ such that 
\begin{equation}
\int_{\Sigma_{s_{i,j}}}\abs{H-\frac{1}{2}\left\la\xx,\nn\right\ra}^2\e^{-\frac{\abs{\xx}^2}{4}}d\mathcal{H}^2<\frac{1}{i}, \quad s_i+\frac{j-1}{i}\leq s_{i,j}\leq s_i+\frac{j}{i}.
\end{equation}
Relabeling $s_{i,j}$, we get an increasing sequence $\lambda_i\to\infty$ with $\lambda_{i+1}/\lambda_i\to 1$ such that 
\begin{equation}\label{eqn:intsplit}
\frac{\lambda_i^2}{4}\int_{\Sigma-\lambda_i\vv_0}\left\la\vv_0,\nn\right\ra^2\e^{-\frac{\abs{\xx}^2}{4}}d\mathcal{H}^2
=\int_{\Sigma-\lambda_i\vv_0}\abs{H-\frac{1}{2}\left\la\xx,\nn\right\ra}^2\e^{-\frac{\abs{\xx}^2}{4}}d\mathcal{H}^2\to 0.
\end{equation}
Here we use the self-shrinker equation \eqref{eqn:shrinker} in the above equality.

In order to apply Theorem \ref{thm:sheeting}, we derive the total curvature bound and improved estimate of the mean curvature. First the local Gauss-Bonnet estimate of Ilmanen \cite[Theorem 3]{I1} together with \eqref{eqn:intsplit} and \cite[Theorem 1.1]{DX} implies that 
\begin{equation}\label{eqn:totalcurvl}
\int_{\Sigma\cap B_k(\lambda_i\vv_0)}\abs{A}^2d\mathcal{H}^2<C
\end{equation}
for some $C$ depending only on $\mathbf{F}[\Sigma]$, the genus bound of $\Sigma$ and $k$. Second, by a recent result of Song \cite[Theorem 19]{So}, there exists a $C^\prime>0$ such that
\begin{equation}\label{eqn:lcurv}
\abs{A(\xx)}\leq C^\prime\left(1+\abs{\xx}\right)\quad\mbox{for all $\xx\in\Sigma$}.
\end{equation}
To be self-contained, we include Song's proof with minor changes in Appendix B. Now given $\xx_0\in\RR^3\setminus B_2$, we define $\tilde{\Sigma}=|\xx_0|(\Sigma-\xx_0)$. Then 
\begin{equation}
\lap\left\la\vv_0,\nn\right\ra-\frac{1}{2\abs{\xx_0}^2}\left\la\abs{\xx_0}\xx_0+\xx,\grad\left\la\vv_0,\nn\right\ra\right\ra+\abs{A}^2\left\la\vv_0,\nn\right\ra=0 \quad\mbox{on $\tilde{\Sigma}$}.
\end{equation}
Note that by \eqref{eqn:lcurv}, $|A| \leq 2C^\prime$ on $\tilde{\Sigma}\cap B_1$. It follows from \cite[Theorem 8.17]{GT} that 
\begin{equation}\label{eqn:meanvalue}
\sup_{\tilde{\Sigma}\cap B_{\frac{1}{2}}}\abs{\left\la\vv_0,\nn\right\ra}^2<C^{\prime\prime}\int_{\tilde{\Sigma}\cap B_1}\left\la\vv_0,\nn\right\ra^2d\mathcal{H}^2
\end{equation}
for some $C^{\prime\prime}$ depending only on $C^\prime$. Hence \eqref{eqn:intsplit} and \eqref{eqn:meanvalue} imply that
\begin{equation}\label{eqn:splitl}
\sup_{\xx\in\Sigma\cap B_k(\lambda_i\vv_0)}\abs{\xx}^{-1}\abs{H(\xx)}\to 0, \quad\mbox{as $i\to\infty$}.
\end{equation}

Hence, together with \eqref{eqn:totalcurvl} and \eqref{eqn:splitl}, Theorem \ref{thm:sheeting} implies that for each $k>2$, there exist a positive integer $i_k$ and $0<r_k<1$ such that, if $i\geq i_k$, then for all $\xx_0\in\Sigma\cap B_{k/2}(\lambda_i\vv_0)$, the connected component of $\Sigma\cap B_{r_k}(\xx_0)$ containing $\xx_0$ is given by the graph of a function over a subset of $T_{\xx_0}\Sigma$ with gradient bounded by $1/k$. Therefore, letting $R_i=k/4$ for $i_k\leq i<i_{k+1}$, the theorem follows immediately from the Arzela-Ascoli theorem, the $L^2$ asymptotic splitting estimate \eqref{eqn:intsplit} and the classification \cite{AL} of $1$-dimensional self-shrinkers (see also \cite[Corollary 10.45]{CM2}).
\end{proof}

\begin{remark}
Alternatively, invoking the properness and finiteness of genus, one may use \cite[Lemma 4]{Si} of Simon and maximum principle (in a similar manner as in Appendix B) to prove a weaker version of Theorem \ref{thm:seqsheeting}, which implies that each component of $(\Sigma-\lambda_i\vv_0)\cap B_{R_i}$ converges with multiplicity one in the sense of varifolds to a plane or a self-shrinking cylinder. Nonetheless, notice that our proof of Theorems \ref{thm:sheeting} and \ref{thm:seqsheeting} is free from \cite[Lemma 4]{Si}.
\end{remark}

In the asymptotically cylindrical case, we use a topological argument to show that Theorem \ref{thm:seqsheeting} holds true for every sequence $\lambda_i\to\infty$. Furthermore, as self-shrinking cylinders are $\mathbf{F}$-unstable, it follows that the multiplicity of the limiting self-shrinking cylinder is equal to one, proving Corollary \ref{cor:cylinder}.

\begin{proof}[Proof of Corollary \ref{cor:cylinder}]
Without loss of generality, we assume that $\vv_0=\ee=(1,0,0)$. First Theorem \ref{thm:seqsheeting} implies that there exists an integer $N$ and two sequences $R_i\to\infty$ and $\lambda_i\to\infty$ with $\lambda_{i+1}/\lambda_i\to 1$ such that for each $i$, 
\begin{equation}
\Sigma\cap B_{R_i}(\lambda_i\ee)=\bigcup_{j=1}^N\Sigma_{i,j},
\end{equation}
where each $\Sigma_{i,j}$ is connected and $\Sigma_{i,j}\cap\Sigma_{i,k}=\emptyset$ if $j\neq k$. Moreover, for each $j$, $\Sigma_{i,j}-\lambda_i\ee$ converges in the locally $C^1$ topology to the self-shrinking cylinder $\mathcal{C}$ with axis parallel to $\ee$. Moreover, we may assume that $\lambda_{i+1}-\lambda_i\to\infty$ as $i\to\infty$.

Next we define 
\begin{equation}
V_r=\left\{\xx=(x_1,x_2,x_3)\in\RR^3: x_2^2+x_3^2<r^2\right\}\quad\mbox{for $r>0$}.
\end{equation}
Then by the Clearing-out Lemma of Brakke \cite{Ba}, given $R>2$, there exists a $\mathcal{X}$ sufficiently large such that the surface
\begin{equation}
\Sigma^\prime=\Sigma\cap V_R\cap\left\{x_1>\mathcal{X}\right\}
\end{equation}
satisfies $\partial\Sigma^\prime\subset\{x_1=\mathcal{X}\}$. We show that $\Sigma^\prime$ has a finite number of ends. 
 
For each $i$ and $j$, we define 
\begin{equation}
\gamma_{i,j}^{\mathcal{R}}=\Sigma_{i,j}\cap\left\{x_1=\lambda_i+R\right\}\quad\mbox{and}\quad\gamma_{i,j}^{\mathcal{L}}=\Sigma_{i,j}\cap\left\{x_1=\lambda_i-R\right\}.
\end{equation}
For all $i$ large, $\gamma_{i,j}^{\mathcal{R}}$ and $\gamma_{i,j}^{\mathcal{L}}$ are simple closed curves. Let us consider the surface
\begin{equation}
\Sigma_i=\Sigma^\prime\cap\left\{\lambda_i+R<x_1<\lambda_{i+1}-R\right\}.
\end{equation}
Then
\begin{equation}
\partial\Sigma_i=\bigcup_{j=1}^N\left(\gamma_{i,j}^{\mathcal{R}}\cup\gamma_{i+1,j}^{\mathcal{L}}\right).
\end{equation}
Given $\Sigma_{i,0}$ a connected component of $\Sigma_i$, assume that 
\begin{equation}
\partial\Sigma_{i,0}=\left(\bigcup_{k=1}^{K}\gamma_{i,j^\prime_k}^{\mathcal{R}}\right)\cup\left(\bigcup_{l=1}^{L}\gamma_{i+1,j_l}^{\mathcal{L}}\right).
\end{equation}
Thus by the local monotonicity formula \cite[Section 10]{Wh}, $K\geq L$ for all $i$ large. On the other hand, by \cite[Lemma 3.20]{CM2}, 
\begin{equation}
\lap x_1-\frac{1}{2}\left\la\xx,\grad x_1\right\ra+\frac{1}{2}x_1=0 \quad\mbox{on $\Sigma$}.
\end{equation}
Thus by the maximum principle, $\partial\Sigma_{i,0}$ cannot be contained in $\{x_1=\lambda_{i+1}-R\}$. Hence $K=L$ and this further implies that for $\mathcal{X}$ sufficiently large,
\begin{equation}
\Sigma^\prime=\bigcup_{j=1}^{N^\prime}\Sigma_j^\prime,\quad\Sigma_j^\prime\cap\Sigma_k^\prime=\emptyset\quad\mbox{if $j\neq k$},
\end{equation}
where for each $j$, both $\Sigma_j^\prime$ and $\Sigma_j^\prime\cap\{x_1>\lambda_i\}$ for all $\lambda_i>\mathcal{X}$ are connected.

For simplicity, let us assume that $N^\prime=1$ for this paragraph, i.e., $\Sigma^\prime=\Sigma_1^\prime$. Otherwise, apply the argument below to each $\Sigma_j^\prime$. If $N>1$, there exists an infinite subsequence $i_p$ of positive integers such that for each $p$, there is a connected component $\Sigma_{i_p,0}$ of $\Sigma_{i_p}$ satisfying that 
\begin{equation}
\partial\Sigma_{i_p,0}=\left(\bigcup_{k=1}^{K_p}\gamma^{\mathcal{R}}_{i_p,j^\prime_k}\right)\cup\left(\bigcup_{l=1}^{L_p}\gamma^{\mathcal{L}}_{i_p+1,j_l}\right)
\end{equation}
with $K_p=L_p\geq 2$. For $l=1,2$, take a point $\xx_l$ on $\gamma^{\mathcal{L}}_{i_p+1,j_l}$. Then $\xx_1$ and $\xx_2$ can be joined by an embedded curve $\gamma_p$ in $\Sigma_{i_p}$. On the other hand, as $\Sigma^\prime\cap\{x_1>\lambda_{i_p+1}\}$ is connected, points $\yy_1\in\gamma^{\mathcal{R}}_{i_p+1,j_1}$ and $\yy_2\in\gamma^{\mathcal{R}}_{i_p+1,j_2}$ can be joined by another embedded curve $\tilde{\gamma}_p$ in $\{x_1>\lambda_{i_p+1}+R\}$. Finally $\xx_l$ and $\yy_l$ for $l=1,2$ can be joined by an embedded curve 
\begin{equation}
\gamma_p^l\subset\Sigma_{i_p+1,j_l}\cap\left\{\lambda_{i_p+1}-R<x_1<\lambda_{i_p+1}+R\right\}.
\end{equation} 
Thus $\gamma_p\cup\gamma_p^1\cup\gamma_p^2\cup\tilde{\gamma}_p$ is a simple closed curve which does not separate $\Sigma$. This immediately implies that $\Sigma$ has infinite genus, giving a contradiction. 

Hence the local regularity theorem \cite{WhReg} implies that for $1\leq j\leq N^\prime$, as $\lambda\to\infty$, $(\Sigma_j^\prime-\lambda\ee)\cap B_R$ converges locally smoothly to $\mathcal{C}\cap B_R$ of multiplicity one.

To conclude the proof, we show that $N^\prime=1$. For each $j$, we define a family of surfaces by
\begin{equation}
\Sigma^\prime_{j,s}=\Sigma^\prime_j-\e^{\frac{s}{2}}\ee.
\end{equation}
Then $\{\Sigma^\prime_{j,s}\}_{s>0}$ is a normalized mean curvature flow, i.e., for $\xx\in\Sigma^\prime_{j,s}$,
\begin{equation}
\left\la\partial_s\xx,\nn\right\ra=-H+\frac{1}{2}\left\la\xx,\nn\right\ra
\end{equation}
Moreover there exists an $S$ sufficiently large such that $\{\Sigma^\prime_{j,s}\}_{s>S}$ in $B_{R/2}$ is given by the ${\rm graph}(u_j(\cdot,s))$ over $\mathcal{C}$ and $u_j(\cdot,s)\to 0$ in the locally $C^\infty$ topology as $s\to\infty$. 

Assume that $N^\prime>1$. Since $\Sigma$ is embedded, we order $u_j$ such that $u_1<\cdots<u_{N^\prime}$. Thus, letting $v=u_{N^\prime}-u_1$, $v>0$ satisfies 
\begin{equation}\label{eqn:linermcf}
\partial_sv=\lap v-\frac{1}{2}\left\la\xx,\grad v\right\ra+v+Q,
\end{equation}
where 
\begin{equation}
Q=\sum_{i,j=1}^2a_{ij}\partial_i\partial_jv+\sum_{i=1}^2b_i\partial_iv+cv,
\end{equation}
$a_{ij}, b_i, c$ are functions depending on $u_1$ and $u_{N^\prime}$, and $\partial_1$ and $\partial_2$ denote taking partial derivative with respect to $x_1$ and to the spherical coordinate $\theta$ on $\mathcal{C}$, respectively. We leave the derivation of \eqref{eqn:linermcf}, which is a small perturbation of the linearization of the normalized mean curvature flow on $\mathcal{C}$, to Appendix A. 

Let $\phi:\RR^3\to [0,1]$ be a cut-off function such that $\phi=1$ in $B_{R/4}$, $\phi=0$ outside $B_{R/2}$, and $|\grad\phi|<8/R$. For simplicity, we will omit $d\mathcal{H}^2$ in the estimates below. Multiplying \eqref{eqn:linermcf} by $v^{-1}\phi^2\e^{-|\xx|^2/4}$ and integrating over $\mathcal{C}$, it follows from integration by part and the Cauchy-Schwarz inequality that
\begin{equation}\label{eqn:timederive}
\begin{split}
\frac{d}{ds}\int_{\mathcal{C}}\log v\,\phi^2\e^{-\frac{\abs{\xx}^2}{4}}
\geq & \int_{\mathcal{C}}\left(\phi^2-2\abs{\grad\phi}^2\right)\e^{-\frac{\abs{\xx}^2}{4}}+\int_{\mathcal{C}}v^{-1}Q\, \phi^2\e^{-\frac{\abs{\xx}^2}{4}}\\
& +\frac{1}{2}\int_{\mathcal{C}}\abs{\grad\log v}^2\phi^2\e^{-\frac{\abs{\xx}^2}{4}}.
\end{split}
\end{equation}
By a similar argument, we have that
\begin{equation}
\int_{\mathcal{C}}v^{-1}Q\, \phi^2\e^{-\frac{\abs{\xx}^2}{4}}\geq-\int_{\mathcal{C}}\left(B_0\abs{\grad\log v}^2\phi^2+B_1\phi^2+B_2\abs{\grad\phi}^2\right)\e^{-\frac{\abs{\xx}^2}{4}},
\end{equation}
where $B_0, B_1, B_2$ are positive functions with $B_0,B_1,B_2\to 0$ as $s\to\infty$, depending on $a_{ij}, \grad a_{ij}, b_i,c$ and $R$. Thus for $R,S$ sufficiently large, \eqref{eqn:timederive} gives 
\begin{equation}
\frac{d}{ds}\int_{\mathcal{C}}\log v\,\phi^2\e^{-\frac{\abs{\xx}^2}{4}}>0\quad\mbox{for all $s>S$},
\end{equation}
which contradicts that $v\to 0$ as $s\to\infty$. Therefore, by the arbitrariness of $R$, this justifies the multiplicity one as claimed in the corollary.
\end{proof}

In the end, invoking Ilmanen's local Gauss-Bonnet estimate again, the boundedness of the second fundamental form of the self-shrinkers in Theorem \ref{thm:bmc} is deduced from first applying Theorem \ref{thm:sheeting} followed by Lemmas \ref{lem:radext} and \ref{lem:imposcnormal}.

\begin{proof}[Proof of Theorem \ref{thm:bmc}]
Since $\Sigma$ is properly embedded with locally finite genus \eqref{eqn:genus} and bounded mean curvature \eqref{eqn:bmc}, it follows from Ilmanen's local Gauss-Bonnet estimate \cite[Theorem 3]{I1} and \cite[Theorem 1.1]{DX} that there exists a $C>0$ depending only on $g_0$, $\mathcal{H}_0$ and $\mathbf{F}[\Sigma]$ such that for all $\xx_0\in\RR^3$,
\begin{equation}\label{eqn:totalcurv2}
\int_{\Sigma\cap B_1(\xx_0)}\abs{A}^2d\mathcal{H}^2<C.
\end{equation}
Thus, applying Theorem \ref{thm:sheeting} for $\kappa_0=C$ and $\delta_0\in (0,1)$ defined by
\begin{equation}
\delta_0^2=\frac{4}{\left(2-\epsilon_2^2\right)^2}-1,
\end{equation}
there exist constants $r_0\in (0,1)$ and $R_0>0$, depending only on $C$, $\epsilon_2$ and $\mathcal{H}_0$, such that for all $\xx_0\in\Sigma$ with $|\xx_0|>R_0$, the connected component of $\Sigma\cap B_{r_0}(\xx_0)$ containing $\xx_0$ is given by the graph of a function over $T_{\xx_0}\Sigma$ with gradient bounded by $\delta_0$. By the definition of $\delta_0$, this implies 
\begin{equation}
\sup_{\substack{\xx\in\Sigma \\ d_\Sigma(\xx,\xx_0)<r_0}}\abs{\nn(\xx)-\nn(\xx_0)}<\epsilon_2.
\end{equation}
Hence, if $\xx_0\in\Sigma$ with 
\begin{equation}
\abs{\xx_0}>\max\left\{2R_0,\frac{2\mathcal{H}_0}{\epsilon_2},\frac{4}{\epsilon_2r_0},\frac{4}{\eta_0r_0}\right\}, 
\end{equation}
applying Lemmas \ref{lem:radext} and \ref{lem:imposcnormal} to $\Sigma\cap B_{r_1}(\xx_0)$ for $r_1=r_0^2|\xx_0|/16$, it follows that $|A(\xx_0)|\leq C^\prime$ for some $C^\prime$ depending only on $\eta_0$, $r_0$, $C_2$ and $C_3$. Therefore, in view of the smooth compactness theorem for self-shrinkers \cite[Theorem 0.2]{CM3} (see also \cite[Proposition 3.4]{BW}), we conclude that $|A|$ is uniformly bounded by a constant depending only on $g_0$, $\mathcal{H}_0$ and $\mathbf{F}[\Sigma]$.
\end{proof}

\appendix

\section{}
We give details of the derivation of the evolution equation \eqref{eqn:linermcf} for $v$. Let $\mathcal{C}$ be the self-shrinking cylinder with axis parallel to $(1,0,0)$. And let $(x,\theta)$ be the cylindrical coordinates on $\mathcal{C}$. Suppose that $u(x,\theta,s)$ is a function on a subset of $\mathcal{C}\times\RR$. We define the normal graph of $u(\cdot,s)$ over $\mathcal{C}$ by
\begin{equation}
\Psi(x,\theta,s)=\left(x,\left(u(x,\theta,s)+\sqrt{2}\right)\cos\theta,\left(u(x,\theta,s)+\sqrt{2}\right)\sin\theta\right).
\end{equation}

First we calculate the first derivatives of $\Psi$:
\begin{align}
\partial_s\Psi & =\partial_su\left(0,\cos\theta,\sin\theta\right),\\
\partial_x\Psi & =\left(1,\partial_xu\cos\theta,\partial_xu\sin\theta\right),\\
\partial_\theta\Psi & =\left(u+\sqrt{2}\right)\left(0,-\sin\theta,\cos\theta\right)+\partial_\theta u\left(0,\cos\theta,\sin\theta\right).
\end{align}
Thus the unit normal $\nn(x,\theta,s)$ is parallel to
\begin{equation}
\NN=\partial_x\Psi\times\partial_\theta\Psi=\left(u+\sqrt{2}\right)\left(\partial_xu,-\cos\theta,-\sin\theta\right)+\partial_\theta u\left(0,-\sin\theta,\cos\theta\right).
\end{equation}
And the inverse of the induced metric by $\Psi(\cdot,s)$ is given by
\begin{equation}
\mathbf{g}^{-1}=\frac{1}{G}\left(
\begin{array}{cc}
\abs{u+\sqrt{2}}^2+\abs{\partial_\theta u}^2 & -\partial_xu\, \partial_\theta u\\
-\partial_xu\, \partial_\theta u & 1+\abs{\partial_x u}^2
\end{array}
\right),
\end{equation}
where
\begin{equation}
G=\abs{\NN}^2=\left(u+\sqrt{2}\right)^2\left(1+\abs{\partial_xu}^2\right)+\abs{\partial_\theta u}^2.
\end{equation}

Furthermore we compute the second derivatives of $\Psi$:
\begin{align}
\partial_x^2\Psi& =\partial_x^2u\left(0,\cos\theta,\sin\theta\right),\\
\partial_x\partial_\theta\Psi & =\partial_x\partial_\theta u\left(0,\cos\theta,\sin\theta\right)+\partial_xu\left(0,-\sin\theta,\cos\theta\right),\\
\partial_\theta^2\Psi & =\left(\partial_\theta^2u-u-\sqrt{2}\right)\left(0,\cos\theta,\sin\theta\right)+2\partial_\theta u\left(0,-\sin\theta,\cos\theta\right).
\end{align}

If the family ${\rm graph}(u(\cdot,s))$ is a normalized mean curvature flow, then
\begin{equation}
\left\la\partial_s\Psi,\NN\right\ra=-H\abs{\NN}+\frac{1}{2}\left\la\Psi,\NN\right\ra.
\end{equation}
We compute and simplify each term in the above equation. Note that
\begin{equation}
-H\abs{\NN}={\rm trace}(\mathbf{g}^{-1}\mathbf{A}),
\end{equation}
where the matrix $\mathbf{A}$ is given by
\begin{equation}
\begin{split}
& \mathbf{A}=\left(
\begin{array}{cc}
\left\la\partial_x^2\Psi,\NN\right\ra & \left\la\partial_x\partial_\theta\Psi,\NN\right\ra \\
\left\la\partial_x\partial_\theta\Psi,\NN\right\ra & \left\la\partial_\theta^2\Psi,\NN\right\ra
\end{array}
\right)\\
= & \left(
\begin{array}{cc}
-(u+\sqrt{2})\partial_x^2u & -(u+\sqrt{2})\partial_x\partial_\theta u+\partial_xu\, \partial_\theta u\\
-(u+\sqrt{2})\partial_x\partial_\theta u+\partial_xu\, \partial_\theta u & (u+\sqrt{2})(u+\sqrt{2}-\partial_\theta^2u)+2\abs{\partial_\theta u}^2
\end{array}
\right).
\end{split}
\end{equation}
Thus a straightforward simplification gives that
\begin{equation}
-H\abs{\NN}=-\left(u+\sqrt{2}\right)\left(\partial_x^2u+\frac{\partial_\theta^2u}{2}+\frac{u}{2}-\frac{1}{\sqrt{2}}\right)+\frac{P}{G},
\end{equation}
where $P$ is an at least quadratic polynomial of $u$ and its derivatives up to second order. On the other hand, we have that
\begin{equation}
\left\la\partial_s\Psi,\NN\right\ra=-\left(u+\sqrt{2}\right)\partial_su,\quad\left\la\Psi,\NN\right\ra=-\left(u+\sqrt{2}\right)\left(u+\sqrt{2}-x\partial_xu\right).
\end{equation}
Hence the equation for $u$ is given by
\begin{equation}
\partial_su=\partial_x^2u+\frac{\partial_\theta^2u}{2}-\frac{x\partial_xu}{2}+u-\frac{P}{\left(u+\sqrt{2}\right)G},
\end{equation}
where the leading part is the linearization of normalized mean curvature flow on $\mathcal{C}$.

Therefore, substituting $u$ by $u_1$ and $u_{N^\prime}$, their difference gives the equation for $v$, that is,
\begin{equation}
\partial_sv=\lap v-\frac{1}{2}\left\la\xx,\grad v\right\ra+v+Q,
\end{equation}
where $\grad$ and $\lap$ are the gradient and Laplacian on $\mathcal{C}$ respectively,
\begin{equation}
Q=a_{11}\partial_x^2v+2a_{12}\partial_x\partial_\theta v+a_{22}\partial_{\theta}^2 v+b_1\partial_xv+b_2\partial_\theta v+cv,
\end{equation}
$a_{ij},b_i,c$ are functions depending on $u_1$, $u_{N^\prime}$ and their derivatives up to second order, and $a_{ij},b_i,c\to 0$ in the locally $C^1$ sense as $s\to\infty$.

\section{}
We follow closely the strategy of Song \cite{So} to show the linear growth of the second fundamental form for properly embedded self-shrinkers with finite genus. 

\begin{theorem}\label{thm:lcurv}
Suppose that $\Sigma\subset\RR^3$ is a complete noncompact properly embedded self-shrinker with finite genus. Then there exists a $C_0>0$ such that
\begin{equation}
\abs{A(\xx)}\leq C_0\left(1+\abs{\xx}\right)\quad\mbox{for all $\xx\in\Sigma$}.
\end{equation}
\end{theorem}

\begin{remark}
In \cite[Theorem 19]{So} some curvature concentration bound of $\Sigma$ is assumed. However, this can be removed by the assumptions on $\Sigma$ and Ilmanen's local Gauss-Bonnet estimate. 
\end{remark}

The rest of this appendix is devoted to proving Theorem \ref{thm:lcurv}. First we show a maximum principle for self-shrinkers, which is a special case of \cite[Proposition 6]{So}. 

Given a unit vector $\vv_0$, we define
\begin{equation}
\mathcal{C}\left(\vv_0\right)=\left\{\xx\in\RR^3: \dist(\xx,\spa(\vv_0))\leq 2,\left\la\xx,\vv_0\right\ra\geq 0\right\},
\end{equation}
and $\PP(\vv_0)$ the plane through $\oo$ with normal $\vv_0$. And we denote by $\mathcal{S}$ the self-shrinking sphere in $\RR^3$.

\begin{lemma}\label{lem:maximum}
Given $\vv_0\in\partial B_1$, suppose that $\Sigma$ is a properly embedded self-shrinker in $\mathcal{C}(\vv_0)$ with $\partial\Sigma\subset\partial\mathcal{C}(\vv_0)$ which satisfies that $\Sigma$ intersects $\partial\mathcal{C}(\vv_0)\setminus\PP(\vv_0)$ transversally, then $\Sigma\cap\bar{B}_2\neq\emptyset$.
\end{lemma}

\begin{proof}
We argue by contraction. Assume that $\Sigma\cap\bar{B}_2=\emptyset$. First we define the function $h:\mathcal{S}\cap\mathcal{C}(\vv_0)\to\RR$ by
\begin{equation}
h\left(\xx\right)=\min\left\{\left\la\yy-\xx,\vv_0\right\ra: \yy\in\Sigma, \Pi(\yy)=\Pi(\xx)\right\},
\end{equation}
where $\Pi$ is the orthogonal projection onto $\PP(\vv_0)$. In our convention, $h(\xx)=\infty$ if there does not exist $\yy\in\Sigma$ with $\Pi(\yy)=\Pi(\xx)$. Thus $h>0$ is well-defined by the properness of $\Sigma$ and our assumption.

Since $\Sigma$ intersects $\partial\mathcal{C}(\vv_0)\setminus\PP(\vv_0)$ transversally, it follows from the properness of $\Sigma$ that $h$ achieves its minimum at some interior point $\xx_0$ with $h(\xx_0)=\la\yy_0-\xx_0,\vv_0\ra$ for some $\yy_0\in\Sigma$ and $\Pi(\yy_0)=\Pi(\xx_0)$. Thus, in intrinsic neighborhoods of $\xx_0$ and $\yy_0$, $\mathcal{S}$ and $\Sigma$ can be written as graphs of functions $f_{-}$ and $f_{+}$, respectively, over $\PP(\vv_0)$, and the difference $f_{+}-f_{-}$ has a local minimum $h(\xx_0)$ at $\Pi(\xx_0)$. However, both $f_{+}$ and $f_{-}$ satisfy the graphical self-shrinker equation and thus by the maximum principle for elliptic equations, the difference $f_{+}-f_{-}$ does not have any positive interior local minimum. This is a contradiction, proving the lemma. 
\end{proof}

In what follows, we use blow-up techniques and maximum principles to prove Theorem \ref{thm:lcurv} by contradiction.

\begin{proof}[Proof of Theorem \ref{thm:lcurv}]
It suffices to prove that there exists a $C>0$ such that for every $\xx_0\in\RR^3\setminus B_1$, letting $r_0=1/|\xx_0|$,
\begin{equation}
\sup_{\xx\in\Sigma\cap B_{r_0}(\xx_0)}\left(r_0-\abs{\xx-\xx_0}\right)^2\abs{A(\xx)}^2\leq C.
\end{equation}

We will argue by contradiction. Namely, if not, by a similar argument as the proof of Lemma \ref{lem:slcurv}, there exist two sequences $\xx_i,\yy_i$ of points in $\RR^3$ with $|\xx_i|\to\infty$ and $\yy_i\in\Sigma\cap B_{r_i}(\xx_i)$ for $r_i=1/|\xx_i|$ such that for $\sigma_i=(r_i-|\xx_i-\yy_i|)/2$,
\begin{equation}
\sigma_i\abs{A(\yy_i)}\to\infty\quad\mbox{and}\quad\sup_{\xx\in\Sigma\cap B_{\sigma_i}(\yy_i)}\abs{A(\xx)}\leq 2\abs{A(\yy_i)}.
\end{equation}
Let us define $\tilde{\Sigma}_i=|A(\yy_i)|(\Sigma-\yy_i)$. Then 
\begin{equation}\label{eqn:curv}
\abs{A(\oo)}=1\quad\mbox{and}\quad\abs{A}\leq 2\quad\mbox{on $\tilde{\Sigma}_i\cap B_{\tilde{R}_i}$},
\end{equation} 
where $\tilde{R}_i=\sigma_i|A(\yy_i)|\to\infty$. On the other hand, as $\sigma_i|\xx_i|<1$, $|\xx_i|/|A(\yy_i)|\to 0$, and it follows from the self-shrinker equation \eqref{eqn:shrinker} that
\begin{equation}\label{eqn:mczero}
\sup_{\xx\in\tilde{\Sigma}_i\cap B_{\tilde{R}_i}}\abs{H(\xx)}\to 0.
\end{equation}

Since $\Sigma$ is properly embedded, by \cite[Theorem 1.1]{DX} and \cite[Lemma 7.10]{CM2}, there exists a $C^\prime>0$ depending only on $\mathbf{F}[\Sigma]$ such that
\begin{equation}\label{eqn:arearatio}
\sup_{\xx\in\RR^3,r>0}r^{-2}{\rm Area}(\Sigma\cap B_r(\xx))\leq C^\prime
\end{equation}
Moreover, together with the genus bound of $\Sigma$ and the self-shrinker equation \eqref{eqn:shrinker}, the local Gauss-Bonnet estimate \cite[Theorem 3]{I1} implies that for all $\xx_0\in\RR^3\setminus B_1$ and $r_0=1/|\xx_0|$,
\begin{equation}
\int_{\Sigma\cap B_{r_0}(\xx_0)}\abs{A}^2d\mathcal{H}^2\leq\kappa,
\end{equation}
where constant $\kappa$ depends only on $C^\prime$ and the genus bound of $\Sigma$. Thus by the scaling invariance of the total curvature of surfaces, we have that
\begin{equation}\label{eqn:totalcurv3}
\int_{\tilde{\Sigma}_i\cap B_{\tilde{R}_i}}\abs{A}^2d\mathcal{H}^2\leq\kappa\quad\mbox{for all $i$}.
\end{equation}

Hence, combining \eqref{eqn:curv}, \eqref{eqn:mczero}, \eqref{eqn:arearatio} and \eqref{eqn:totalcurv3}, it follows from the elliptic regularity theory and the Arzela-Ascoli theorem that, passing to a subsequence and relabeling, the connected component $\tilde{\Sigma}_{i,0}$ of $\tilde{\Sigma}_{i}\cap B_{\tilde{R}_i}$ containing $\oo$ converges locally smoothly to a complete properly embedded nonflat minimal surface $\mathcal{M}$ of genus $0$ and finite total curvature. Therefore, appealing to the classification theorem \cite{LR}, we conclude that $\mathcal{M}$ must be a Catenoid which, without loss of generality, rotates around the straight line of direction $\ee=(1,0,0)$.

To conclude the proof, we use maximum principles for self-shrinkers to rule out the Catenoid. Let $\bar{\Sigma}$ be a connected component of $\Sigma\setminus B_{R}$ such that $\bar{\Sigma}$ has zero genus and, passing to a subsequence and relabeling, the $\yy_i$ are in $\bar{\Sigma}$ and $\yy_i/|\yy_i|\to\vv$ for some $\vv\in\partial B_1$. And let $\gamma_i$ be the simple closed curve in $\Sigma$ such that the nearest points projection of the curve $\tilde{\gamma}_i=|A(\yy_i)|(\gamma_i-\yy_i)$ onto $\mathcal{M}$ is the geodesic simple loop which encircles the neck of $\mathcal{M}$. Since $\Sigma$ is proper, so is $\bar{\Sigma}$ and thus $\bar{\Sigma}\cap B_{2R}$ has a finite number $L$ connected components.  

Let us first consider the case that $|\la\vv,\ee\ra|=1$. Note that $\gamma_i$ for each $i$ large is a convex curve in a plane normal to $\ee$ which separates $\bar{\Sigma}$. And for $i$ large, $\yy_i$ is almost parallel to $\ee$ and $\tilde{\Sigma}_{i,0}$ can be written as the small normal graph over the Catenoid $\mathcal{M}$. Thus by elementary geometry and the Sard theorem, there exist $N,\tilde{R}$ sufficiently large such that if $i>N$, there exist $L$ distinct unit vectors $\vv_{i,1},\dots,\vv_{i,L}$ such that for each $1\leq j\leq L$, $\Sigma$ intersects $\partial\mathcal{C}(\vv_{i,j})\setminus\PP(\vv_{i,j})$ transversally, $\mathcal{C}(\vv_{i,j})\cap\gamma_i=\emptyset$, and $\partial\mathcal{C}(\vv_{i,j})$ intersects both connected components $\Sigma_{i}^{+},\Sigma_{i}^{-}$ of $\Sigma_i\setminus\gamma_i$, where
\begin{equation}
\Sigma_i=\abs{A(\yy_i)}^{-1}\left(\tilde{\Sigma}_{i,0}\cap B_{\tilde{R}}\right)+\yy_i.
\end{equation}
Moreover, for each $i>N$, there exists $R_i>|\yy_i|$ such that if $1\leq j<k\leq L$, the distance between $\mathcal{C}(\vv_{i,j})\setminus B_{R_i}$ and $\mathcal{C}(\vv_{i,k})\setminus B_{R_i}$ is at least $1$. 

Now we choose $i_l>N$ for $1\leq l\leq L$ such that $|\yy_{i_{l+1}}|>R_{i_l}$. Observe that for each $l$, there is an integer $j_l\in [1,L]$ such that $\mathcal{C}(\vv_{i_l,j_l})\cap\gamma_{i_k}=\emptyset$ for all $k\geq l$. As $\bar{\Sigma}$ has zero genus, $\bar{\Sigma}\setminus(\bigcup_{l=1}^L\gamma_{i_l})$ has $L+1$ connected components, each of which must intersect $B_{2R}$ by Lemma \ref{lem:maximum}. This implies that $\bar{\Sigma}\cap B_{2R}$ has at least $L+1$ connected components, giving a contradiction. 

Next we deal with the case that $|\la\vv,\ee\ra|<1$. By our assumptions, there exist $N^\prime,\tilde{R}^\prime$ sufficiently large such that if $i>N^\prime$, there is a half-space $\HH_i$ with outward unit normal $\yy_i/|\yy_i|$ to $\partial\HH_i$ which satisfies that $\HH_i\cap\tilde{\gamma}_i=\emptyset$ and $\HH_i$ intersects both connected components of $(\tilde{\Sigma}_{i,0}\cap B_{\tilde{R}^\prime})\setminus\tilde{\gamma}_i$. Here we use that any Catenoid has two parallel planar ends. Note that if $i>N^\prime$ for $N^\prime$ sufficiently large, $\partial B_r\cap B_{1/|\yy_i|}(\yy_i)$ for $|r-|\yy_i||<1/|\yy_i|$ are well approximated by planes normal to $\yy_i/|\yy_i|$. Thus, as $\bar{\Sigma}$ has zero genus, $\bar{\Sigma}\setminus (\bigcup_{l=1}^L\gamma_{N^\prime+l})$ has $L+1$ connected components, on each of which $|\xx|^2$ achieves its minimum away from $\bigcup_{l=1}^L\gamma_{N^\prime+l}$. On the other hand, 
\begin{equation}
\lap\abs{\xx}^2-\frac{1}{2}\left\la\xx,\grad\abs{\xx}^2\right\ra+\abs{\xx}^2=4\quad\mbox{on $\Sigma$}.
\end{equation}
Thus the maximum principle implies that $|\xx|^2$ has no local minimum in the interior of $\bar{\Sigma}$. Hence each connected component of $\bar{\Sigma}\setminus (\bigcup_{l=1}^L\gamma_{N^\prime+l})$ must intersect $B_{2R}$, from which it follows that $\bar{\Sigma}\cap B_{2R}$ has at least $L+1$ connected components. This is a contradiction. 
\end{proof}


\begin{thebibliography}{99}

\bibitem{AL} U. Abresch and J. Langer, \emph{The normalized curve shortening flow and homothetic solutions}, J. Differential Geom. 23 (1986), 175--196.

\bibitem{BW} J. Bernstein and L. Wang, \emph{A Remark on a uniqueness property of high multiplicity tangent flows in dimension $3$}, Int. Math. Res. Notices. (2014), to appear, arXiv: 1402.6687.

\bibitem{Ba} K. Brakke, The motion of a surface by its mean curvature, Mathematical Notes, 20, Princeton University Press, Princeton, N.J., 1978. 

\bibitem{Be} S. Brendle, \emph{Embedded self-similar shrinkers of genus $0$}, Preprint (2014), arXiv: 1411.4640.

\bibitem{CZ} X. Cheng and D. Zhou, \emph{Volume estimate about shrinkers}, Proc. Amer. Math. Soc. 141 (2013), no. 2, 687--696.

\bibitem{CS} H.I. Choi and R. Schoen, \emph{The space of minimal embeddings of a surface into a three-dimensional manifold of positive Ricci curvature}, Invent. Math. 81 (1985), no. 3, 387--394.

\bibitem{Ch} D. Chopp, \emph{Computation of self-similar solutions for mean curvature flow}, Experiment. Math. 3 (1994), no. 1, 1--15.

\bibitem{CIM} T.H. Colding, T. Ilmanen, and W.P. Minicozzi II, \emph{Rigidity of generic singularities of mean curvature flow}, Publications IHES (2015), to appear, arXiv: 1304.6356.

\bibitem{CM1} T.H. Colding and W.P. Minicozzi II, \emph{Sharp estimates for mean curvature flow of graphs}, J. Reine Angew. Math. 574 (2004), 187--195.

\bibitem{CM2} T.H. Colding and W.P. Minicozzi II, \emph{Generic mean curvature flow I: generic singularities}, Ann. of Math. (2) 175 (2012), no. 2, 755--833. 

\bibitem{CM3} T.H. Colding and W.P. Minicozzi II, \emph{Smooth compactness of self-shrinkers}, Comment. Math. Helv. 87 (2012), no. 2, 463--475. 

\bibitem{DX} Q. Ding and Y.L. Xin, \emph{Volume growth, eigenvalue and compactness for self-shrinkers}, Asian J. Math. 17 (2013), no. 3, 443--456.

\bibitem{EH1} K. Ecker and G. Huisken, \emph{Mean curvature evolution of entire graphs}, Ann. of Math. (2) 130 (1989), no. 3, 453--471. 

\bibitem{EH2} K. Esker and G. Huisken, \emph{Interior estimates for hypersurfaces moving by mean curvature}, Invent. Math. 105 (1991), no. 3, 547--569.

\bibitem{GT} D. Gilbarg and N.S. Trudinger, Elliptic partial differential equations of second order, Reprint of the 1998 edition, Classics in Mathematics, Springer-Verlag, Berlin, 2001.

\bibitem{H} G. Huisken, \emph{Local and global behaviour of hypersurfaces moving by mean curvature}, Differential geometry: partial differential equations on manifolds (Los Angeles, CA, 1990), 175--191, Proc. Sympos. Pure Math., 54, Part 1, Amer. Math. Soc., Providence, RI, 1993.

\bibitem{I1} T. Ilmanen, \emph{Singularities of mean curvature flow of surfaces}, Preprint (1995).

\bibitem{I2} T. Ilmanen, \emph{Lectures on mean curvature flow and related equations}, Preprint (1998).

\bibitem{KKM} N. Kapouleas, S.J. Kleene, and N.M. M{\o}ller, \emph{Mean curvature self-shrinkers of high genus: Non-compact examples}, J. Reine Angew. Math. (2012), to appear, arXiv: 1106.5454.

\bibitem{KM} S.J. Kleene and N.M. M{\o}ller, \emph{Self-shrinkers with a rotational symmetry}, Trans. Amer. Math. Soc. 366 (2014), no. 8, 3943--3963.

\bibitem{LR} F.J. L\'{o}pez and A. Ros, \emph{On embedded minimal surfaces of genus zero}, J. Differential Geom. 33 (1991), no. 1, 293--300.

\bibitem{MW} O. Munteanu and J. Wang, \emph{Geometry of shrinking Ricci solitons}, Compos. Math. (2015), to appear, arXiv: 1410.3813.
 
\bibitem{Ng} X.H. Nguyen, \emph{Construction of complete embedded self-similar surfaces under mean curvature flow, part III}, Duke Math. J. 163 (2014), no. 11, 2023--2056.

\bibitem{Si} L. Simon, \emph{Existence of surfaces minimizing the Willmore functional}, Comm. Anal. Geom. 1 (1993), no. 2, 281--326.

\bibitem{So} A. Song, \emph{A maximum principle for self-shrinkers and some consequences}, Preprint (2014), arXiv: 1412.4755.

\bibitem{Wa1} L. Wang, \emph{A Bernstein type theorem for self-similar shrinkers}, Geom. Dedicata 151 (2011), 297--303.

\bibitem{Wa2} L. Wang, \emph{Uniqueness of self-similar shrinkers with asymptotically conical ends}, J. Amer. Math. Soc. 27 (2014), no. 3, 613--638. 

\bibitem{Wa3} L. Wang,\emph{Uniqueness of self-similar shrinkers with asymptotically cylindrical ends}, J. Reine Angew. Math. (2014), to appear.

\bibitem{Wh} B. White, \emph{Stratification of minimal surfaces, mean curvature flows, and harmonic maps}, J. Reine Angew. Math. 488 (1997), 1--35.

\bibitem{WhReg} B. White, \emph{A local regularity theorem for mean curvature flow}, Ann. of Math. (2) 161 (2005), no. 3, 1487--1519.

\end{thebibliography}
\end{document}